\documentclass[a4paper,11pt]{amsart} %

\usepackage{geometry}
\usepackage{amsmath, amsthm}
\usepackage{amssymb, mathtools}
\usepackage{extarrows}
\usepackage{mathabx}

\usepackage{relsize}
\usepackage[dvipsnames]{xcolor}

\usepackage{xr}
\usepackage{bigfoot}

\usepackage[shortlabels]{enumitem}

\usepackage{todonotes}

\usepackage{accents}

\usepackage{hyperref}

\usepackage{tikz-cd}

\usepackage{endnotes}

\usepackage{booktabs}
\usepackage{tabulary}

\usepackage{soul}
\usepackage{scalerel}

\usepackage[utf8]{inputenc}
\usepackage[greek,english]{babel}
\usepackage{lmodern}

\newtheorem{THEOREM}{Theorem}[section]

\newtheorem{Conclusion}[THEOREM]{Conclusion}

\newtheorem{Theorem}[THEOREM]{Theorem}
\newenvironment{theorem}{\begin{Theorem}}{\end{Theorem}}

\newtheorem{Lemma}[THEOREM]{Lemma}
\newenvironment{lemma}{\vskip6pt\begin{Lemma}}{\end{Lemma}\vskip6pt}

\newtheorem{Proposition}[THEOREM]{Proposition}
\newenvironment{proposition}{\vskip12pt\begin{Proposition}}{\end{Proposition}\vskip6pt}

{\newtheoremstyle{definition}
  {12pt}
  {0pt}
  {\upshape\color{black}}
  {}
  {\bfseries\color{black}}
  {.}
  {.5em}
  {}
\theoremstyle{definition}

\newtheorem{notation}[THEOREM]{Notation}
\newtheorem{definition}[THEOREM]{Definition}

\newtheorem{remark}[THEOREM]{Remark}

\newtheorem{question}[THEOREM]{Question}
\newtheorem*{assumption}{Important Assumption}

\newtheorem*{note-nono}{Note}
\newtheorem{note-no}{Note}

}

\newtheorem{Claim}[THEOREM]{Claim}

\newtheorem{Subclaim}[THEOREM]{Subclaim}

\newtheorem{Corollary}[THEOREM]{Corollary}
\newenvironment{corollary}{\begin{Corollary}}{\end{Corollary}}

\newtheoremstyle{myrmkstyle}
  {12pt}
  {3pt}
  {\upshape\color{blue}}
  {}
  {\bfseries\color{blue}}
  {.}
  {.5em}
  {}

  \theoremstyle{myrmkstyle}

\newtheoremstyle{myrmkstyleblk}
  {12pt}
  {0pt}
  {\upshape\color{black}}
  {}
  {\bfseries\color{black}}
  {.}
  {.5em}
  {}

  \theoremstyle{myrmkstyleblk}

\newtheoremstyle{privatermkstyle}
  {12pt}
  {3pt}
  {\upshape\color{blue}}
  {}
  {\bfseries\color{blue}}
  {}
  {0em}
  {}

  \theoremstyle{privatermkstyle}

\newtheoremstyle{privatermkstylered}
  {3pt}
  {3pt}
  {\upshape\color{red}}
  {}
  {\bfseries\color{red}}
  {}
  {0em}
  {}

  \theoremstyle{privatermkstylered}

\newenvironment{eqn}{\begin{equation*}}{\end{equation*}}
\newenvironment{itmz}{\begin{itemize}[itemsep=6pt,topsep=-3pt,leftmargin=0.15in]}{\end{itemize}}

\def\>{\rangle}
\def\<{\langle}

\newcommand{\squ}{\scalebox{0.7}{$\square$}}

\newcommand{\squaresub}[1]{\square_{\ssss #1}{\hskip2pt}}

\newcommand{\rge}{\mathop{\mathrm{rge}}} 
\newcommand{\dom}{\mathop{\mathrm{dom}}} 
\newcommand{\Var}{\mathop{\mathrm{Var}}} 
\renewcommand{\lim}{\mathop{\mathrm{lim}}} 
\renewcommand{\ni}{\notin}
\renewcommand{\to}{\mathop{\parbox{.5cm}{\rightarrowfill}}}

\newcommand{\lh}{\mathop{\mathrm{lh}}}

\newcommand{\ceu}{\mathop{\mathrm{ceu}}}
\newcommand{\On}{\mathop{\mathrm{On}}}
\newcommand{\Card}{\mathop{\mathrm{Card}}}

\newcommand{\psh}[1]{\mathop{\mathrm{pSh}}({#1})}

\newcommand{\psho}{\mathop{\mathrm{pSh}}({\Omega})}

\newcommand{\pshol}{\mathop{\mathrm{pSh}}({\Omega},L)}
 
\newcommand{\concat}{\kern-.25pt\raise4pt\hbox{$\frown$}\kern-.25pt}
\newcommand{\on}{\upharpoonright}  

\newcommand{\Nbb}{\mathbb{N}}

\newcommand{\forces}{\Vdash}
\newcommand{\isom}{\cong}
\newcommand{\thinks}{\models}

\newcommand{\Vee}{\bigvee}
\newcommand{\Wedge}{\bigwedge}

\newcommand{\UA}{{\mathcal U}{\mathcal A}}
\def\sss{\hskip2pt}

\def\ssss{\hskip1pt}
\def\tupof#1#2{\<\ssss #1\sss:\sss #2\ssss\>} 
\def\tup#1{\<\ssss #1\ssss\>} 
\def\setof#1#2{\{\ssss #1\sss:\sss #2\ssss\}} 
\def\set#1{\{\ssss #1\ssss\}} 

\newcommand{\cardrule}{\hrule height.2pt}
\newcommand{\cardrulefill}{\cleaders\cardrule\hfill}
\newcommand{\cardchar}[1]{\vbox{\ialign{##\crcr
    \cardrulefill\crcr\noalign{\kern1pt\nointerlineskip}
    $\hfil\displaystyle{#1}\hfil$\crcr}}}
\newcommand{\barchar}[1]{\vbox{\ialign{##\crcr
    \cardrulefill\crcr\noalign{\kern1.5pt\nointerlineskip}
    $\hfil\displaystyle{#1}\hfil$\crcr}}}
\def\card#1{\cardchar{\cardchar{#1}}}

\newcommand{\forkindep}[2][]{%
  \mathrel{
    \mathop{
      \vcenter{
        \hbox{\oalign{\noalign{\kern-.3ex}\hfil$\vert$\hfil\cr
              \noalign{\kern-.7ex}
              $\smile$\cr\noalign{\kern-.3ex}}}
      }
    }\displaylimits^{#2}_{#1}
  }
}

\title{Some contributions to presheaf model theory, II - back and forth}

\author{Andreas Brunner, Charles Morgan and Darllan Concei\c{c}\~ao Pinto}

\address{Departamento de Matem\'atica, Universidade Federal de Bahia,
  Avenida Milton Santos s/n, S/N, Ondina, Salvador, Bahia, Brazil. CEP: 40170-110.}
\email{andreas@dcc.ufba.br}
\email{charlesmorgan@ufba.br}
\email{darllan@ufba.br}

\begin{document}

\begin{abstract} We discuss the back and forth technique in the context of presheaf model theory. The essence of the back and forth technique lies in showing the relationship
  between various hierarchies which calibrate similarity between two models and, more generally, between two pairs consisting of a model and a tuple from it.
  In this paper we define several such hierarchies for presheaf models (and tuples of sections from them): those based on the degree of extendibility of partial isomorphisms through literal back and forth conditions,
  on sharing specific, abstract invariants which we define (the $F^\alpha_{M,\bar{a}}$ of \S{}\ref{function_analysis} for example), on agreeing on the (truth) values
  of instantiations of formulae up to a given amount of quantifier completity, on the existence of winning strategies for player II in certain Ehrenfeucht-Fra\"iss\'e-type games
  and, finally, on satisfying certain infinitary sentences that arise in the construction of Scott sentences. We ultimately show that all of these hierarchies align. 
\end{abstract}

\maketitle

\section*{Introduction}\label{intro}

Back and forth arguments appear to
have first arisen in the literature in proofs given by Huntingdon (\cite{Huntingdon}) and Hausdorff (\cite{Hausdorff})
of Cantor's theorem that any two countable dense linear orders without endpoints are isomorphic. The technique was
elaborated and refined in the work of Fra\"iss\'e, Ehrenfeucht, Scott, Karp and Tiamanov, amongst others, in the 1950s and 1960s
(\cite{Fra50},\cite{Fra53},\cite{Fra56},\cite{Scott},\cite{Karp},\cite{Taimanov1},\cite{Taimanov2}). It
was given a central r\^ole in the presentation of the fundamentals of model theory by Poizat (\cite{Poizat-fr}, \cite{Poizat-en}).
It remains an active research theme to this day.\footnote{$\sss$See, for example, V\"a\"an\"anen's fascinating book (\cite{Vaananen}).
Harrison-Trainor's acute recent survey (\cite{Harrison-Trainor}) views the theme from a different, more complexity theoretic, slant.}

The essential idea of a back and forth argument is that one tries to extend, iteratively, a partial map $f$ between two structures $M$ and $N$ 
which shows similarity between the substructures generated by $\dom(f)$ and $\rge(f)$, by first selecting an element, say $x$, of $M\setminus \dom(f)$
and trying to `match' it with an element, say $y$, of $N\setminus \rge(f)$, and then selecting an element $z$ of $N\setminus \rge(f)\cup\set{y}$
and trying to match it with some $w$ in $M\setminus\dom(f)\cup\set{x}$, and then iterating the process.

Presheaf model theory has a history going back to 1970s, some of which we sketched in the introduction to
\cite{BMPI}. In this paper we continue working in what has crystalised, currently, as the most fruitful framework for presheaf model theory,
that of Fourman-Scott and Higgs (\cite{FS}, \cite{Higgs1}, \cite{Higgs2}). 
One fixes a complete Heyting algebra, $\Omega$, and, given a language $L$,
uses the elements of $\Omega$ as `truth' values, values for the interpretations of instantiations of $L$-formulae in $L$-structures,
rather than $\set{\bot,\top}$ as one does classically. In this setting one has, inter alia, versions of
the downward L\"owenheim-Skolem theorem (\cite{M88}), the omitting types theorem (\cite{BM04}), the generic model theorem (\cite{Berrio}),
Los’s theorem (\cite{M90}, \cite{Ara}), Robinson’s method of diagrams (\cite{BM14}, \cite{BMPI}), the completness theorem and a limited form of
the model existence theorem (\cite{B00}), as well as results in neostability theory (\cite{BMPI}).

In view of this progress on various aspects of presheaf model theory it is natural to wonder 
whether the back and forth method is applicable in presheaf model theory and if so how it might look.\footnote{$\sss$One should stress
it is not at all a given that all methods and results in classical model theory have successful analogues in
presheaf model theory (similarly, for example, to the situation for continuuous model theory).
On the one hand, one knows some things definitely do not work in good analogy to classical model theory -- for example the compactness theorem fails.
On the other, there are things in classical model theory for which it remains unknown whether there are good analogues in presheaf model theory  --
for example the Feferman-Vaught theorem.}
Inital efforts in this direction were made by Sette-Caicedo (\cite{SC}) and Brunner (\cite{B00}), who proved versions of
Fra\"iss\' e's theorem for sheaves of models over topological spaces and, generalizing \cite{SC},
for sheaves over arbitrary complete Heyting algebras, respectively. 

The proofs of Fra\"iss\' e style theorems both in \cite{EFT}, \S{}XII.3 or \cite{Marker}, for example,\S{}2.4, for Fra\"iss\' e's original theorem,
and in \cite{SC} and \cite{B00}, have interesting features in common. We discuss them here in order to give some motivation for our presentation below.

A high-level overview of the proofs is that they take the following form. Let $L$ be a first order relational language. 
If $M$, $N$ are $L$-structures one can define a hierarchy $\tupof{Q^n_{M,N}}{n\in \Nbb}$ of refinements  
of the set of partial isomorphisms between $M$ and $N$. The tiers of the hierarchy are governed by the amount of back and forth extensions
it is possible to make to the constituent partial isomporhisms. Then, for each $n\in \Nbb$, $M$ and $N$ satisfy the same sentences of
quantifier degree at most $n$ if and only if $Q^n_{M,N}$ is non-empty.\footnote{$\sss$We have hidden what a partial isomorphism is,
what we mean by quantifier degree and so on. The reader's intuitions will be more than sufficient at this level of generality.}

The first thing to remark is, although not stated in this somewhat abstract way in any account of which we are aware,
the proofs in effect inductively build for each $L$-structure $M$ a hierarchy of invariants. Let us call them $\tupof{Z^n_M}{n\in \Nbb}$.
Marker (\cite{Marker}) calls a specific instance of this process \emph{Scott-Karp analysis} and we adopt this name for the process generally.\footnote{$\sss$We only use
the notations $Q^n_{M,N}$ and $Z^n_M$ in this introduction. In the body of the paper the sets $Q_\alpha(p)$ generalize the r\^ole of the $Q^n_{MN}$ -- the $M$, $N$ being implicit.
The function defined in \S{}\ref{function_analysis} and the sentence defined in \S{}\ref{the_phi_alpha_Ma} are examples of the hierarchies of invariants.}

The proofs for any $n\in\Nbb$ then consist of a circle of implications. If for all $m\le n$ one has $Z^m_M=Z^m_N$ then $Q^n_{M,N}$ is non-empty. If 
$Q^n_{M,N}$ is non-empty then $M$ and $N$ satisfy the same sentences of
quantifier rank at most $n$. Finally, if $M$ and $N$ satisfy the same sentences of
quantifier rank at most $n$ then for all $m\le n$ one has $Z^m_M=Z^m_N$.

This is striking, because the invariant $Z^n_M$ for any model and for any particular $n\in\Nbb$ is a single, albeit perhaps complicated, object,
so the check whether for all $m\le n$ one has $Z^m_M=Z^m_N$ is a very discrete (indeed, here, finite) process,
whilst satisfying the same sentences up to a certain quantifier degree or
having $Q^n_{M,N}$ non-empty are in some sense global properties of $M$ and $N$ -- in order to see that they hold in both cases one has to check
some property holds of everything of the relevant sort (specifically, sentences and objects in the models).
Consequently, whilst working on this material (and \cite{BMPIV}) our point of view has evolved to the extent that we now regard the equivalence between
having for all $m\le n$ that $Z^m_M=Z^m_N$ and $Q^n_{M,N}$ being non-empty as the true heart of Fra\"iss\' e style theorems.

Secondly, rather than working purely with $L$-structures one has to work (and prove the equivalences) with pairs $(M,\bar{a})$ consisting of
an $L$-structure $M$ and a tuple $\bar{a}$ of elements from it. This is natural since, for example, as soon as one wishes to compare the satisfaction
in two models of sentences $Qx\phi(x)$ with quantifiers one has to use the inductive definition satisfaction and thus look at
instantiations of $\phi(x)$ in the models.


In this paper we give a more comprehensive treatment
of questions around the back and forth method for presheaf model theory.
We remark at the outset that we go beyond the milieu of \cite{B00} in a number of ways.
Firstly, we deal with arbitrary signatures, and do not restrict ourselves to relational languages. Secondly, throughout we deal with transfinite
notions of refinement, equivalence and so on, rather than restricting to finite ones. Thirdly, we organize our definitions, results, proofs and so on,
in such as way as to deal with first order logic and infinitary languages simultaneously.\footnote{$\sss$The results in \cite{B00}, Chapter 3, were stated for sheaves,
however the arguments there do not make use of glueing and the proofs go through for presheaves.}

We briefly discuss the contents of the paper. 
We start, \S{}\ref{part_isoms}, by discussing partial isorphisms -- essentially maps that preserve unnested atomic formulae, proving some of their basic properties
and then, \S{}\ref{Q_alpha_p} defining refinements of this notion under iterated satisfaction of back and forth conditions. We next 
introduce certain Ehrenfeucht-Fra\"iss\'e-style games for pairs of presheaf models (strictly speaking for pairs of pairs consisting of a presheaf model and a tuple from it)
and show the existences of winning strategies for player $II$ in these games is
equivalent to the existence of sufficiently refined partial isomorphisms, \S{}\ref{games}.

We then, \S{}\ref{function_analysis}, immediately turn to defining hierarchies of invariants for pairs consisting of a presheaf model and a tuple from it. The invariants here
are certain functions, see Definition (\ref{introducing_the_functions_tuples}). We then show
a `true heart of Karp theorem': two pairs consisting of a presheaf model and a tuple having common invariants up to any specific ordinal
is equivalent to there being partial isomorphisms refined to the same ordinal level between them.
We then define a Scott-style rank, prove it is well-defined and discuss limitations on its usefulness (as far as we can see) in comparison with the classical case.

Up to this point we have only needed the notion of an (unnested) atomic formula. However, as we wish to consider to what extent
two pairs consisting of a presheaf model and a tuple from it satisfy the same formulae in general it is now necessary to introduce languages
and the quantifier degree of formulae, \S{}\ref{languages}.
Since we work with languages which are not necessarily relational there may be terms in the languages which arise from the application
function symbols to other terms and these give rise to \emph{nested} formulae. The modified rank refines quantifier degree by taking into account
the complexity of terms involved in a formula.

It turns out for technical reasons we need to work in \S{}\ref{equiv_implies_equality} and \S{}\ref{Towards_Karp's_theorem}
with \emph{unnested} formulae. Consequently, we have to discuss, \S{}\ref{unnesting}, how to unnest formulae in such a way that the interpretations of the instantiations of the unnested formula
remain unchanged from those for the original formula. The process of defining an unnesting is common to classical and presheaf model theory, and results in
the quantifier degree of the unnesting being equal to the modified rank of the original formula.
In classical model theory proving the interpretations are unchanged is straightward, although it is touched on only rarely.
An exception is \cite{Hodges}, where a `proof by example' is given. (We give a more complete account in \cite{BMPIV}.)
However, for presheaf model theory the process of proving the interpretations of instantiations
do not change when passing from a formula to its unnesting is considerably more delicate.
Consequently, Theorem (\ref{unnesting_pres_value_atom_fml}) and Corollary (\ref{evaluation_persists_through_unnesting}),
which show the process does in fact work, are important results. 
  
We then, \S{}\ref{additional_connectives}, discuss how to augment our languages by adding a unary
connective $\squaresub{p}$ for each element $p$ of $\Omega$. We call
the resultant language $L^{\squ}$.  We introduce $L^{\squ}$ because we
need to work with structures and sentences over it in order to prove
the main result of \S{}\ref{equiv_implies_equality}, that if
interpretations of instantiations of formulae up to a certain ordinal
degree coincide for two pairs consisting of a presheaf model and a
tuple from it then the hierarchies of functions defined from each of
the pairs are equal up to that ordinal.\footnote{We include commentary
here on an alternative approach through adding nullary connectives in
the the style of continuous model theory.}

A methodological comment on this maneouvre is in order. In \cite{BM14} the unary connectives $\squaresub{p}$ were employed in
some results and proofs on Robinson's method of diagrams for presheaves.
In \cite{BMPI}, \S{}3, we were able to give sharpened versions of those results whilst avoiding
using the unary connectives.\footnote{At the end of \S{}\ref{additional_connectives} we
briefly discuss how one can present the results from \cite{BMPI} on the method of diagrams in aesthetically different,
  to some readers more pleasing, ways using $L^{\squ}$.}
Here we were not able to avoid using them and believe it is not
possible to do so.\footnote{$\sss$See Question (\ref{question_can_we_avoid_the_squaresubs}).}
We can, and do, define hierarchies of invariants in \S{}\ref{function_analysis} without using the
connectives.\footnote{$\sss$The Fra\"iss\'e-style theorems in both \cite{SC} and \cite{B00} are also in the context of $L^{\squ}$.
One difference from the approach in this paper is that both require the unary connectives to define invariants.}
However, our proofs of the main results of \S{}\ref{equiv_implies_equality} seem to necessitate them.
This should not necessarily be so surprising. For any ($L^{\squ}$-)formula $\phi$, presheaf model $M$ and tuple $\bar{a}$ from $M$,
the interpretation of $\squaresub{\bot}\phi(a)$ is the same as the interpretation of $\neg\phi(\bar{a})$
(and the interpretation of $\squaresub{\top}\phi(a)$ is simply that of $\phi(\bar{a})$).
Thus, for classical model theory, \emph{i.e.}, when $\Omega = \set{\bot,\top}$, these connectives are already in the language.
Moreover, the use of the $\squaresub{p}$ in our argument in \S{}\ref{equiv_implies_equality} mirrors
one version of the use of  $\neg$ in classical back and forth arguments (see \cite{BMPIV}, \S{}4).

We also note that for \S{}\ref{equiv_implies_equality}, \S{}\ref{portmanteau} and \S{}\ref{the_phi_alpha_Ma} we restrict attention
to signatures for which all constant symbols have extent $\top$. We discuss why this is necessary at the start of
\S{}\ref{equiv_implies_equality}.

In \S{}\ref{Towards_Karp's_theorem} we show sufficiently refined partial isomorphisms preserve instantiations of all formulae
whose modified rank their degree of refinement exceeds,

The results of \S{}\ref{equiv_implies_equality} and \S{}\ref{Towards_Karp's_theorem}
allow us to complete the proof of Proposition (\ref{portmanteau_prop}), showing the equivalence
between various notions of equivalence between two  pairs consisting of a presheaf model and a tuple from it discussed in the paper.
A corollary (Corollary (\ref{generalization_Karp_thm})
is a generalization of Karp's theorem for presheaves of models for first order structures and $L_{\infty\lambda}$-structures.

In the final section, \S{}\ref{the_phi_alpha_Ma}, we define a hierachy of invariants for (expansions of) $L^{\squ}$-structures each of which is a
formula in the expanded language.
This generalizes the approach of \cite{SC} and \cite{B00}.
We show that these invariants are closely connected to the invariants introduced in \S{}\ref{introducing_the_functions_tuples}.
This allows us to add one further equivalence to the list given in Proposition (\ref{portmanteau_prop}) and recover a proof similar in overall style to those
in \cite{EFT}, \S{}XII.3 or \cite{Marker}, \S{}2.4.

\section{Preliminaries}\label{prelims}

\begin{notation} Most of our notation is standard. We use a handful of standard set theoretic conventions. We write
  $\On$ for the class of ordinals and $\Card$ for the class of cardinals.
  Given a set $X$, an ordinal $\alpha$ and a cardinal $\mu$
we denote by $\card{X}$ the cardinality of $X$ and $^{\alpha}X$ the set of sequence of length $\alpha$ from $X$, equivalently
the set of functions from $\alpha$ to $X$, while $^{<\alpha}X = \bigcup_{\beta<\alpha} {^{\beta}X}$, and $[X]^{<\mu}$ is the set of
subsets of $X$ of cardinality less than $\mu$.

We follow the typical model theoretic convention and denote concatenation by juxtaposition.
If $\bar{a}$, $\bar{b}$ are tuples from a structure $M$ and $c\in M$,
  we write $\bar{a}\bar{b}$ for the concatenation $\bar{a}\concat\bar{b}$,
  and $\bar{a}c$ for the concatenation $\bar{a}\concat c$. More
  explicity, if $\bar{a}=\tupof{a_i}{i<\alpha}\in {^{\alpha}M}$,
  $\bar{b}=\tupof{b_j}{j<\beta}\in {^{\beta}M}$ and $c\in M$, write
  $\bar{a}\bar{b}$ for the sequence $\bar{d}=\tupof{d_i}{i<\alpha +
    \beta}$, where for $i<\alpha$ we have $d_i=a_i$ and for $i =
  \alpha +j$ with $j<\beta$ we have $d_i = b_j$, and write $\bar{a}c$
  for the sequence $\bar{e}=\tupof{e_i}{i<\alpha + 1}$ where for
  $i<\alpha$ we have $e_i=a_i$ and $e_\alpha = c$.
\end{notation}

Our notation for presheaves is taken from \cite{BMPI}. This paper also provides the necessary background material on
presheaf model theory and some historical discussion.

Fix for the remainder of the paper a complete Heyting algebra $\Omega$.

\begin{definition}A \emph{signature} is a collection of function, relation and constant symbols. More formally,
  $\tau$ is a signature if it is a quadruple $\tup{\mathcal F,\mathcal R,\mathcal C,\tau'}$, where
  $\mathcal F$, $\mathcal R$ and $\mathcal C$ are pairwise disjoint sets. 
  $\tau':\mathcal F\cup\mathcal R\longrightarrow \Nbb$.
\end{definition}

Fix for the remainder of the paper  a signature $\tau$, $=\tup{\mathcal F,\mathcal R,\mathcal C,\tau'}$.
Until we reach \S{}\ref{equiv_implies_equality} we also assume $\mathcal C$ is a presheaf.

Later in the paper we will construct various languages over $\tau$. The key parameters controlling the constructions
will be three cardinals bounding the number of free variables allowed in formulae, the size of conjunctions and disjuntions
allowed and the size of sets over which quantification is allowed. However, initially all that we need are the notions
of a term and of an atomic formula, and these remain the same for any of these languages
with the same number of free variables allowed in formulae.

Fix a cardinal $\lambda$ for the remainder of this section and the next four sections of the paper
(\S\S{}2-5). This will ensure that notions such
as atomic formulae and $\UA$, defined below, have a clear meaning.

\begin{definition}   Define \emph{terms in a language over $\tau$ in $\lambda$-many variables}
  as follows.
    \begin{enumerate}
    \item For $\alpha<\lambda$ the variable $v_\alpha$  is a term.
      \item Each constant symbol $c\in  \mathcal C$ is a term.
    \item If $f$ is an $n$-ary
    function symbol and $t_0$, \dots, $t_{n-1}$ are terms then so is $f(t_0,\dots,t_{n-1})$.
    \end{enumerate}
\end{definition}

\begin{definition} The \emph{atomic formulae in a language over $\tau$ in $\lambda$-many variables}
  are defined as follows.
     \begin{enumerate}
\item If $t_0$ and $t_1$ are terms, then $t_0=t_1$ is an atomic formula.
\item If $t_0$,\dots, $t_{n-1}$ are terms and $R$ is an $n$-ary relation symbol then\break
   $R(t_0,\dots t_{n-1})$ is atomic formula.
    \end{enumerate}
\end{definition}

\begin{definition}  An {atomic formula in a language over $\tau$ in $\lambda$-many variables}
  is \emph{unnested} if it is of one of the forms
  \[ v_i=v_j,\sss\sss v_i = c,\sss\sss R(v_{i_0},\dots,v_{i_{n-1}})\hbox{ or }v_{i_n} = f(v_{i_0},\dots,v_{i_{n-1}}),\] where the
  $v_i$ are variables, $c$ is a constant symbol, $R$ is an $n$-ary relation symbol and $f$ is an $n$-ary function symbol.
  We write $\UA$ for the collection of these unnested atomic formulae.
\end{definition}

The interpretation of terms in a presheaf $M$ over $\Omega$ is covered
in \cite{BMPI},\S{}2 (painstakingly in view of infelicities in previous
treatments). That interpretation and the interpretation of atomic
formulae go through as in \cite{BMPI} for any $\lambda$, regardless of
what connectives and quantification are allowed. Accordingly, for example,
for $\phi$ atomic and $\bar{a}$ a tuple of sections from $M$, there is
no ambiguity writing $[\phi(\bar{a})]_M$.

\begin{definition}\label{defn_Var_FV} Let $\phi$ be an
  {atomic formula in a language over $\tau$ in $\lambda$-many variables}. Then $\Var(\phi)$ is the set of variables occuring in $\phi$.
\end{definition}

Note, for any such $\phi$ the set $\Var(\phi)$ is finite and all variables in it occur freely in $\phi$.

\section{Partial isomorphisms}\label{part_isoms}

\begin{definition}\label{defn_invariant} Suppose $M$, $N \in\psh{\Omega,L}$, $h:M \longrightarrow N$ is a
  presheaf morphism, $\phi(\bar{v})$ is an atomic formula and $p\in \Omega$.
  We say $\phi(\bar{v})$ is \emph{invariant under} $h$ \emph{in} $p$
  if \[ \hbox{ for all }\bar{a}\in {{}^{\lh(v)}|\dom(h)\on p|}\hbox{ we have }
  [\phi(\bar{a})]_M = [\phi(h``\bar{a})]_N .\]
  We say $\phi(\bar{v})$ is \emph{invariant under} $h$ if it is invariant under $h$ \emph{in} $\top$. 

  We also sometimes say, and this is the usage employed in \cite{BMPI}, that
  $\phi$ is \emph{preserved by} $h$ \emph{in} $p$, resp.~\emph{preserved by} $h$.
\end{definition}

Note that if $\phi(\bar{v})$ is invariant under $h$ it is invariant under $h$ in $p$ for every $p\in\Omega$, since
$\dom(h)\on p \subseteq \dom(h) = \dom(h)\on \top$.

We now define partial isomorphisms, a key ingrediant in back and forth arguments.    


\begin{definition}\label{subpresheaf} Recall from \cite{BMPI},  Definition (1.34) and Proposition (1.37), for $M$ a presheaf over $\Omega$ and $A\subseteq |M|$,
  the subpresheaf of $M$
  generated by $A$, $\tup{A}$, is given by $|\tup{A}| = \setof{a\on
    p}{a\in A \sss\sss\&\sss\sss p\in\Omega}$, $E^{\tup{A}} = E^M\on
  |\tup{A}|$ and $\on^{\tup{A}} = \on^M\on |\tup{A}|$.
\end{definition}

\begin{lemma}\label{ext_fns_to_psh_monos} Suppose $M$, $N$ are presheaves over $\Omega$, $A\subseteq |M|$, $B\subseteq |N|$, and $h:A\longrightarrow B$.
  Suppose   for all $a_0$, $a_1\in A$ we have $[a_0=a_1]_M = [h(a_0)=h(a_1)]_N$. Then we can extend $h$ to a presheaf monomorphism $h:\tup{A}\longrightarrow \tup{B}$
  by for $a\in A$ and $p\in \Omega$  setting $h(a\on p) = h(a)\on p$.
\end{lemma}

\begin{proof} We, firstly, have to show the definition is a good one. So suppose $a_0$, $a_1\in A$, $p\in \Omega$ and $a_0\on p = a_1\on p$. We have
  $Ea_0= Eh(a_0)$ and $Ea_1 = Eh(a_1)$, by the hypothesis on $h$. Further, $Ea_0\on p = Ea_1\on p$. Thus $Eh(a_0)\on p = Eh(a_1)\on p$. Thus
  \begin{eqn}\begin{split} Eh(a_0)\on p = & Eh(a_1)\on p = Ea_0 \on p = [a_0\on p = a_1\on p]_M = [a_0 =a_1]_M\land p\\  = & [h(a_0)=h(a_1)]_N \land p = [h(a_0)\on p =h(a_1) \on p]_N.
    \end{split}
  \end{eqn}Thus  $h(a_0)\on p = h(a_1)\on p $ by the extensionality of $M$. Thus the definition is good. Moreover, we have also now have that $h$ commutes with extents and restrictions, and
  we also have \[ [a_0\on p = a_1\on p]_M = [h(a_0)\on p =h(a_1) \on p]_N = [h(a_0\on p) =h(a_1 \on p)]_N.\] Thus the extended function satisfies the definition for being
  a presheaf monomorphism.
\end{proof}

\begin{definition}\label{L_partial_isom}   Suppose $M$, $N \in\psh{\Omega,L}$, $A\subseteq |M|$, $B\subseteq |N|$, and
  $\tup{A}$, $\tup{B}$ are the subpresheaves of $M$, $N$, respectively, generated by $A$ and $B$.
  A bijection $h:A\longrightarrow B$ regarded as a partial function 
  $h:M\longrightarrow N$ is a \emph{partial isomorphism} (between $M$ and $N$)  if
  for the extension of $h$ to $h:\tup{A}\longrightarrow \tup{B}$ as given in Lemma (\ref{ext_fns_to_psh_monos}), all unnested atomic formulae are invariant under $h$ in $\top$, \emph{i.e.},~
  for all $n$-ary relation symbols $R \in L$, for all $c \in {C}$ and all $n$-ary function symbols $f$ 
  the following formulae  are invariant under $h$ in $\top$.
  \begin{enumerate}[align=parleft,labelsep=1cm, label=(\roman*)]
  \item $v_0 = v_1$,
  \item $R(v_0,\dots,v_{n-1})$
  \item $v=c$
  \item $f(v_0,\dots,v_{n-1})=v_n$
  \end{enumerate}
  Note, $h:\tup{A}\longrightarrow \tup{B}$ is a well defined presheaf monomorphism, by Lemma  (\ref{ext_fns_to_psh_monos}), since
  $v_0=v_1$ is invariant under $h$ in $\top$. Moreover, it satisfies (i)-(iv).
\end{definition}

\begin{notation} If $h:M\longrightarrow N$ is a partial isomorphism and $p\in \Omega$ write
 $M(p)$ for $\setof{a\in M}{Ea = p}$ and $h_p$ for $h\on M(p)$.
\end{notation}



\begin{proposition}\label{useful_prop} Suppose $M$, $N \in\psh{\Omega,L}$,
  $\setof{{\mathfrak c}^M}{c\in C}\subseteq A\subseteq |M|$,
  $\setof{{\mathfrak c}^N}{c\in C}\subseteq B\subseteq N$ and $h:A\longrightarrow B$ preserves extents.
  \begin{enumerate}[align=parleft,labelsep=1cm, label=\upshape(\alph*),rightmargin=4.7pt]
  \item Let $c\in C$. If $v=c$ is invariant under $h$ then $h({\mathfrak c}^M) = {\mathfrak c}^N$.
    \item Let $c\in C$.
    If $v_0=v_1$ is invariant under $h$ and $h({\mathfrak c}^M) = {\mathfrak c}^N$ then
    $v=c$ is invariant under $h$.
  \item Let $c$, $d\in C$. If $v=c$ and $v=d$ are invariant under $h$ then the formula $c=d$ is invariant under $h$.
  \item Let $\bar{c} \in C^m$ and let $R$ be an $(m+n)$-ary relation in $L$. If for each $j<m$ the formula $v=c_j$
    is invariant under $h$ and if the formula $R(v_0,\dots,v_{n-1},u_0,\dots,u_{m-1})$ is invariant under $h$, 
    then the formula  ${R(\bar{v},\bar{c})}$ is invariant under $h$.
  \end{enumerate}
\end{proposition}

\begin{proof}
  (a) Let $c\in C$. Observe that \[E_N{\mathfrak c}^N = E_Cc = E_M {\mathfrak c}^M = E_N h({\mathfrak c}^M) ,\]
       where the first two equalities come from the definition of the interpretation of terms and the third
       equality comes from $h$ being a presheaf morphism.
       On the other hand,
       \begin{eqn}\begin{split}
        E_Cc = E{\mathfrak c}^M = & [{\mathfrak c}^M = {\mathfrak c}^M]_M = [(v = c)({\mathfrak c}^M)]_M = \\
         & [ (v = c)(h({\mathfrak c}^M)]_N = [h({\mathfrak c}^M) = {\mathfrak c}^N]_N 
         \end{split}\end{eqn}where, again the first equality comes from the definition of the interpretation of terms,
          the second equality is from the definition of $[.=.]_M$, the fourth is from the invariance of $v=c$ under $h$,
          and the third and fifth come from the
       definition of the interpretation of atomic formulae. Hence, by extentionality, $h({\mathfrak c}^M)={\mathfrak c}^N$.

   (b) Let $a\in A$.  Using the definition of the interpretation of atomic formulae  for the first and third equalities,
       Definition(\ref{L_partial_isom})(i), the invariance of $v_0=v_1$ under $h$,
       for the second equality and the hypothesis for the final equality, we have
       \begin{eqn}\begin{split}
       [a = {\mathfrak c}^M]_M = & [(v_0 = v_1)(\set{a,{\mathfrak c}^M})]_M = 
       [(v_0 = v_1)(h``\set{a,{\mathfrak c}^M})]_N = \\
       & [ h(a) = h({\mathfrak c}^M)]_N =  [ h(a) = {\mathfrak c}^N]_N
                \end{split}\end{eqn}

    (c) Using, again, the definition of the interpretation of atomic formulae  for the first and third equalities,
       the invariance of $v = d$ under $h$ for the second equality, and the conclusion of (a), which follows
       from the invariance of $v = c$, for the fourth equality,
 \begin{eqn}\begin{split}
           [{\mathfrak c}^M = {\mathfrak d}^M]_M = & [(v = d)({\mathfrak c}^M)]_M =
       [(v = d)(h({\mathfrak c}^M)]_N = \\
      & [h({\mathfrak c}^M) = {\mathfrak d}^N)]_N  = [{\mathfrak c}^N = {\mathfrak d}^N]_N .
                \end{split}\end{eqn}

    (d) Suppose  $R$ is an $(m+n)$-ary relation in $L$,
    $\bar{a}\in |A|^n$ and $\bar{c} \in C^m$. Then
     \[    [R(\bar{a},\bar{\mathfrak{c}}^M)]_M = [R(h``\bar{a},h``\bar{\mathfrak{c}}^M)]_N = 
           [R(h``\bar{a},\bar{\mathfrak{c}}^N)]_N \sss,\]
           applying firstly the invariance of $R(v_0,\dots,v_{n-1},u_0,\dots,u_{m-1})$ under $h$  and then
           the conclusion of (a) applied to each $c_j$ for $j<m$.
\end{proof}

\section{Refinements of partial isomorphisms}\label{Q_alpha_p}


\begin{definition}\label{Q_mu_alpha_p}  For each $p\in \Omega$ we will define a decreasing sequence of sets
  $Q_{\alpha}(p)$, for $\alpha\in \On$, by induction.
   
  Let $Q_0$ be the set of partial isomorphisms between
  elements of $\psh{\Omega,L}$. For all $p\in \Omega$ we set $Q_{0}(p) = Q_0$.

  One could instead for $p\in\Omega$ take
  \[ Q_0(p) = \setof{h\in Q_0}{\bigwedge \setof{Ec}{c\in \dom(h)} = p}.\]
  However, this does not seem to affect the results or their proofs.
  
  Now suppose $\beta<\alpha$ and for all $p\in\Omega$ we have defined $Q_{\beta}(p)$.

  If $\alpha$ is a limit ordinal then we set $Q_{\alpha}(p) = \bigcap_{\beta<\alpha} Q_{\gamma}(p)$.    

  If $\alpha = \beta+1$  
  then $h:M\longrightarrow N\in Q_{\alpha}(p)$ if and only if $h\in Q_{\beta}(p)$ and

    \begin{enumerate}[label=(\alph*)]
  \item  (Forth) For all $\bar{c}=\tupof{c_i}{i<\gamma} \in |M\on p|^{<\mu}$ there is some set $\setof{(h_j,q_j,\bar{d_j})}{j\in J}$ such that for each
    $j\in J$ we have $h_j\in Q_\beta(q_j)$, $h_{q_j} \subseteq h_j$, $d_j = \tupof{d_{ji}}{i<\gamma}\in |N\on E\bar{c}|^{<\mu}$, $Ed_j=q_j$,
    $ \bar{c}\on q_j \subseteq \dom (h_j)$, for all $i<\gamma$ we have $h_j(c_i\on q_j)=d_{ji}$, 
  and $\Vee_{j\in J} q_j = E\bar{c}$.
\vskip12pt
    
\item  (Back) 
  For all $\bar{d}=\tupof{d_i}{i<\gamma} \in |N\on p|^{<\mu}$ there is some set \discretionary{}{}{}
  ${\setof{(h_j,q_j,\bar{c}_j)}{j\in J}}$,  such that for each $j\in J$ we have
  $h_j\in Q_\beta(q_j)$, $h_{q_j} \subseteq h_j$, $ d\on q_j \subseteq \rge (h_j)$, $c_j = \tupof{c_{ji}}{i<\gamma}\in |M\on E\bar{d}|^{<\mu}$,
  $Ec_j=q_j$, for all $i<\gamma$ we have $h_j(c_{ji})=d_i\on q_j$, 
  and $\Vee_{j\in J} q_j = E\bar{d}$.
      \end{enumerate}
\end{definition}

\begin{definition}
Now let  $\bar{a} \in {^{<\lambda} |M\on p|}$, $\bar{b}\in {^{\lh(\bar{a})} |N\on p|}$. Let
  $h:\bar{a}\longrightarrow \bar{b}$ be given by  $h(a_i)=b_i$ for each $i<\lh(\bar{a})$.
  For $\alpha\in \On$ let  $(M,\bar{a})\sim^p_\alpha (N,\bar{b})$ if $h\in Q_{\alpha}(p)$.
\end{definition}

\begin{definition} We write $M\sim_\alpha N$ if $Q_\alpha(\top)$ is non-empty. 
\end{definition}

\section{Games}\label{games}

We can give equivalents to the existence of elements of the $Q_\alpha(p)$ in terms of player II having a winning strategy in a dynamic Ehrenfeucht-Fra\"iss\'e style game.
The adjective dynamic signifies that although each run of the game is finite the length is not predetermined.

\begin{definition} Let $M$, $N\in \pshol$, $\alpha\in \On$, $\bar{a}\in {^{<\mu}|M|}$, $\bar{b}\in {^{\lh(\bar{a})}|N|}$ and let $h:M\longrightarrow N$ be a partial isomorphism
  with $h$, where for
  all $i<\lh(a)$ we have   $h(a_i) = b_i$ (and so $E\bar{a}=E\bar{b}$).
  
  We define $G^\mu_\alpha(M,\bar{a},N,\bar{b},h)$ to be the following game. 
  At move $n$ player I plays $(s_n,\alpha_n)$ where $\bar{s}_n\in {^{<\mu}|M|} \cup {^{<\mu}|N|}$ and $\alpha_n<\alpha$. At this point there will be a function $h^*_n$ determined by
  $h$ and the run of the game to this point.
  Player II responds with a set $\setof{(h_j,q_j,\bar{t}_j)}{j\in J}$.

  The players must obey the following rules. 

  At move $0$, player I must ensure $E\bar{s}_0 \le E\bar{a}$ and choose $\alpha_0<\alpha$. Set $h^*_0=h$. As a notational convenience, but playing no part in the definition
  of the game, set $q^*_0=E\bar{a}$ and $\bar{t}^*_0 = \bar{b}$.

  At move $n$, player II must ensure if $\bar{s}_n\in {^{<\mu}|M|}$ then each $\bar{t}_j\in {^{\lh(\bar{s}_n)}|N|}$ and if $\bar{s}_n\in |N|$ each $\bar{t}_j\in  {^{\lh(\bar{s}_n)}|M|}$.

  Moreover, player II must ensure for each $j\in J_n$ that $E\bar{t}_j=q_j$, $h_j$ is a partial isomorphism
  with $h^*_n \on M(q_j) \subseteq h_j$ and $E\bar{s}_n = \bigvee \setof{E\bar{t}_j}{j\in J_n}$. Player II must also ensure if $\bar{s}_n\in |M|$ then $h_j``\bar{s}_n\on q_j=\bar{t}_j$
  and if $\bar{s}_n\in |N|$ then $h_j``\bar{t}_j = \bar{s}_n\on q_j$.

  At move $n+1$, player I must choose $\alpha_{n+1}<\alpha_n$, select $j\in J_n$, and ensure that for this choice of $j$ we have $E\bar{s}_{n+1} \le E\bar{t}_j$.

  Define, at this point, $h^*_{n+1}$ to be the $h_j$ for this choice of $j$. It is convenient for the proof of Theorem (\ref{thm_equiv_ws_Qs})
  to also set $q^*_{n+1}$ to be $q_j$ and $\bar{t}^*_{n+1}$ to be the $t_j$ for this choice of $j$, although this notation convenience plays no r\^ole in the definition of the game.

  A player loses a run of the game wins when they cannot make a legal move, at which point the other player is the winner.
\end{definition}

Clearly since part of the constraints on player I's legal moves is that $\alpha>\alpha_0>\alpha_1\dots$ each run of the game ends in a finite number of moves. 

\begin{theorem}\label{thm_equiv_ws_Qs} Let $\alpha\in\On$, $M$, $N\in \pshol$, $\bar{a}\in {^{<\mu}|M|}$, $\bar{b}\in {^{\lh(\bar{a})}|N|}$ and let
  $h:M\longrightarrow N$ be a partial isomorphism, where
  we have   $h``\bar{a} = \bar{b}$ (and so $E\bar{a}=E\bar{b}$). Assume that $|M|\cap |N|=\emptyset$. 
  Then $h$ shows $(M,\bar{a})\sim^{E\bar{a}}_\alpha (N,\bar{b})$ if and only if
  player II has a winning strategy for the game $G^\mu_\alpha(M,\bar{a},N,\bar{b},h)$.
\end{theorem}

\begin{proof} (Very similar to \cite{Vaananen}, Proposition 7.17.) ``$\Longrightarrow$''. We define a winning strategy $\tupof{\sigma_n}{n<\omega}$ for player II. Player II's strategy will ensure that
  in any run of the game for each $n$ we have $h^*_n \in  Q_{\alpha_n+1}(q^*_n)$.
  Suppose player II has followed the strategy in response to player I's first $n$ moves and, as their $n+1$th move player I then
  plays $\alpha_n$, chooses $j\in J_{n-1}$ if $n>0$ and plays $\bar{s}_n$. Since the $\alpha$s are decreasing we have $h^*_n\in Q_{\alpha_n+1}(q^*_n)$.
  We set $\sigma_n((\alpha_0,\bar{s}_0),\dots,(\alpha_n,\bar{s}_n))$ to be a set $\setof{(h_j,q_j,\bar{t}_j)}{j\in J}$ as given by the forth property or the back
  property for $h^*_n$, as applicable, when applied to $\bar{s}_n$.  This is clearly a winning strategy.

  ``$\Longleftarrow$''. Suppose $\sigma=\tupof{\sigma_n}{n<\omega}$ is a winning strategy for player II. We show by induction that $h\in Q_{\alpha}(E\bar{a})$.
  This is so for $\alpha=0$ by the hypothesis that $h$ is a partial isomorphism. For limit $\beta\le \alpha$ we are done by induction and the definition of $h$ as being an element of $Q_\beta(E\bar{a})$
  if and only if it is an element of each $Q_\gamma(E\bar{a})$ for $\gamma<\beta$. Finally, suppose $\beta<\alpha$ and $h\in Q_\beta(E\bar{a})$. For any $\bar{s}\in {^{<\mu}|M|}\cup {^{<\mu}|N|}$
  consider the play of the game where at move $0$ player $1$ plays the move $(\beta,s)$. Since $\sigma$ is a winning strategy, player II has a legal response to this movc and
  the response witnesses either the forth property or the back property for $\bar{s}$ depending on whether $\bar{s}$ was taken from $|M|$ or $|N|$. Thus $h\in Q_{\beta+1}(E\bar{a})$.
\end{proof}

\section{Scott-Karp analysis for presheaves of models without sentences}\label{function_analysis}

In this section we give a `Scott-Karp analysis'\footnote{$\sss$Following \cite{Marker}'s usage.} for
presheaves of models. A central point is that we define certain sequences of functions $F^{\alpha}_{M,\bar{a}}$, $G^{\alpha}_{M,N,\bar{a}}$ and
$H^{\alpha}_{M,\bar{a}}$, for ordinals $\alpha$, which will serve as hierarchies of invariants in the sense outlined in the
introduction.\footnote{Infinitary logic is commonly used in the formulation of invariants in Scott-Karp analysis. However we avoid this here.
We also avoid using the unary connectives employed. in \cite{SC} and \cite{B00}.}  As in any Scott-Karp-type analysis, the point, eventually, as will be seen in
Proposition (\ref{portmanteau_prop}), is that
for any two pairs $(M,\bar{a})$ and $(N,\bar{b})$ the longer the sequences of invariants agree the greater the similarity between the structures
(in terms of agreeing on instantiations of formulae of greater and greater modified rank), and \emph{vice versa}.

\begin{notation}   If $\alpha\in \On$, $X$ and $Z$ are sets or even proper classes,  and
  $f:X \times \alpha \longrightarrow Z $, and $\beta\le \alpha$ write
  $f\on \beta$ for the function $g:X\times \beta\longrightarrow Z$ given by for each $x\in X$ and $\gamma<\beta$ setting
  $g(x,\gamma) = f(x,\gamma)$.

  In the second notation the restriction is only on the second coordinate. That we use the same symbol here as for restrictions in presheaves should in
  practice cause no confustion.
\end{notation}

\begin{definition}\label{introducing_the_functions_tuples}  We will define for
  $M$, $N\in \pshol$, $\bar{a}\in {^{<\mu}|M|}$ and $\alpha\in\On$ functions $F^{\alpha}_{M,\bar{a}}$, $G^{\alpha}_{M,N,\bar{a}}$ and
  $H^{\alpha}_{M,\bar{a}}$, each with codomain $\Omega$, by induction on $\alpha$.



The definition of the $H^{\alpha}_{M,\bar{a}}$ is arguably the easiest to assimilate - at least in terms of its domain being readily comprehensible. However, the domain
of $H^{\alpha}_{M,\bar{a}}$, for $\alpha>0$, is a proper class. The $F^{\alpha}_{M,\bar{a}}$ are a variation on the $H^{\alpha}_{M,\bar{a}}$ with more complex,
recursively defined, yet set-sized domains, which some readers may favour in assuaging foundational qualms. The $G^{\alpha}_{M,N,\bar{a}}$ are a further variation,
allowing us to simplify some arguments, however at the cost of weakening the results provable.

Let $M\in \pshol$ and  $\bar{a}\in {^{\lh(\bar{a})}|M|}$. Let $I_{M,\bar{a}}$ be the function given by
setting \begin{eqn}\begin{split}
      \dom(I_{M,\bar{a}}) =  \setof{(\psi,f)}{\psi \in Y\sss\sss\&\sss\sss \exists \ssss n & \in \omega \sss\sss \Var(\psi)= \set{v_0,\dots,v_{n-1}} \\
         &\sss\sss\&\sss\sss
    f:N\longrightarrow \lh(\bar{a})} ,
\end{split}\end{eqn}and for $(\psi,f)\in\dom(I_{M,\bar{a}})$ seting
\[ I_{M,\bar{a}}(\psi,f)  = [\psi(a_{f(0)},a_{f(1)},\dots a_{f(n-1)}))]_M.\]

(We can think of the `$I$' as indicating ``interpretation'', of $\psi$ in $M$ at $\bar{a}$, is afoot.)

Let $M$, $N\in \pshol$ and $\bar{a}\in {^{<\lambda}|M|}$. Set
  \[ F^0_{M,\bar{a}} = G^0_{M,N,\bar{a}} = H^0_{M,\bar{a}} = I_{M,\bar{a}} . \]

  (So we could also think of the letter `$I$' as indicating a collection of  ``initializers'' of the sequences 
   of $F^{\alpha}_{M,\bar{a}}$, $G^{\alpha}_{M,N,\bar{a}}$ and $H^{\alpha}_{M,\bar{a}}$ for $\alpha\in \On$.)

Let $M$, $N\in \pshol$, $\bar{a}\in {^{<\lambda)}|M|}$ and $0<\alpha\in\On$.  The elements of the domains
  of $F^{\alpha}_{M,\bar{a}}$, $G^{\alpha}_{M,N,\bar{a}}$ and $H^{\alpha}_{M,\bar{a}}$ will consist of pairs (in the case of $F^{\alpha}_{M,\bar{a}}$)

  or triples (in the cases of  $G^{\alpha}_{M,N,\bar{a}}$ and $H^{\alpha}_{M,\bar{a}}$) whose last element is some ordinal less than $\alpha$.

  For all $\beta<\alpha$ we will have
  \[ F^{\alpha}_{M,\bar{a}}\on \beta = F^{\beta}_{M,\bar{a}}, \sss\sss G^{\alpha}_{M,N,\bar{a}}\on \beta = G^{\beta}_{M,N,\bar{a}}
\hbox{ and } H^{\alpha}_{M,\bar{a}}\on \beta = H^{\beta}_{M,\bar{a}} .\]

Thus in each case for limit $\alpha$ the functions are completely determined by induction by the definitions for $\beta<\alpha$.

It remains to define the various functions for successor ordinals. Let
  $M$, $N\in \pshol$, $\bar{a}\in {^{<\lambda}|M|}$ and $\alpha\in\On$.

We, first of all, define $H^{\alpha+1}_{M,\bar{a}}(.,.,\alpha)$.
We set \begin{eqn}\begin{split}
      \dom(H^{\alpha+1}_{M,\bar{a}}(.,.,\alpha)) = \setof{(K,\bar{s}\bar{t}, \alpha)}{K\in & \pshol, \\ \sss \bar{s}\in {^{\lh(\bar{a})}|K|},
       & \bar{t}\in {^{<\mu}|K|} \sss\sss\&\sss\sss E\bar{t} \le E\bar{a}\land E\bar{s}} .
  \end{split}
\end{eqn}


  If $(K,\bar{s}\bar{t},\alpha) \in \dom(H^{\alpha+1}_{M,\bar{a}})$ then
  \[ H^{\alpha+1}_{M,\bar{a}}(K,\bar{s}\bar{t},\alpha) = \bigvee \setof{E\bar{c}}{\bar{c}\in {^\gamma|M|}
    \sss\sss\&\sss\sss E\bar{c} \le E\bar{t} \sss\sss\& \sss\sss H^{\alpha}_{K,\bar{s}\bar{t}\on Ec} = H^{\alpha}_{M,\bar{a}\bar{c}}}  .\]

We, next, define $G^{\alpha+1}_{M,N,\bar{a}}(.,.,\alpha)$.
We set \begin{eqn}\begin{split}
      \dom(G^{\alpha+1}_{M,N\bar{a}}(.,.,\alpha)) = \setof{(K,\bar{s}\bar{t}, \alpha)}{K\in & \set{M,N}, \\ \sss \bar{s}\in {^{\lh(\bar{a})}|K|},
       & \bar{t}\in {^{<\mu}|K|} \sss\sss\&\sss\sss E\bar{t} \le E\bar{a}\land E\bar{s}} .
  \end{split}
\end{eqn}


If $(K,\bar{s}\bar{t},\alpha) \in \dom(G^{\alpha+1}_{M,N,\bar{a}})$ then if $K=M$ set $L=N$ and if $K=N$ set $L=M$.

If $(K,\bar{s}\bar{t},\alpha) \in \dom(G^{\alpha+1}_{M,N,\bar{a}})$ then
\[ G^{\alpha+1}_{M,N,\bar{a}}(K,\bar{s}\bar{t},\alpha) =
\bigvee \setof{E\bar{c}}{\bar{c}\in |M| \sss\sss\&\sss\sss E\bar{c} \le E\bar{t} \sss\sss\& \sss\sss G^{\alpha}_{K,L,\bar{s}\bar{t}\on E\bar{c}} = G^{\alpha}_{M,N,\bar{a}\bar{c}} }  .\]

Note that the definition is not symmetrical -- the order of the $M$ and $N$ in the subscripts matters.
(Of course, the $\bar{a}$ being a tuple coming from $M$ already
shows asymmetry.)

Note also that whilst in the definition of $H^{\alpha+1}_{M,\bar{a}}$ the $K$ was an arbitrary element of the proper class $\pshol$, in the definition of
$G^{\alpha+1}_{M,N,\bar{a}}$ the $K$ is limited to being one of $M$ and $N$.

We, finally, define $F^{\alpha+1}_{M,\bar{a}}(.,\alpha)$.
We set \begin{eqn}\begin{split}
      \dom(F^{\alpha+1}_{M,\bar{a}}(.,.,\alpha)) = \setof{(x, \alpha)}{\exists K\in & \pshol, \sss \bar{s}\in {^{\lh(\bar{a})}|K|},\\ 
       & \bar{t}\in {^{<\mu}|K|} \sss\sss\&\sss\sss E\bar{t} \le E\bar{a}\land E\bar{s} \sss\sss\&\sss\sss x = F^{\alpha}_{K,\bar{s}\bar{t}} } .
  \end{split}
\end{eqn}


  If $K\in\pshol$ and
  $\bar{s}\bar{t}\in {^{\lh(\bar{a})+\gamma}|K|}$ with $E\bar{t} \le E\bar{a}\land E\bar{s}$ then
  \[ F^{\alpha+1}_{M,\bar{a}}(F^\alpha_{K,\bar{s}\bar{t}},\alpha) =
  \bigvee \setof{E\bar{c}}{\bar{c}\in {^{\lh(\bar{t})}|M|} \sss\sss\&\sss\sss E\bar{c} \le E\bar{t} \sss\sss\& \sss\sss F^{\alpha}_{K,\bar{s}\bar{t}\on E\bar{c}} = F^{\alpha}_{M,\bar{a}\bar{c}}} .\]
\end{definition}

Note, since $\card{{{^\UA}{\Omega}}} = \card{\Omega}^{\lower.4em\card{\textsuperscript{$\UA$}}}$, by induction $\dom(F^{\alpha+1}_{M,\bar{a}})$ is a set.

\begin{lemma}\label{fn_range_bounded_tuples} In each case the range of $F^{\alpha}_{M,\bar{a}}$, $G^{\alpha}_{M,N,\bar{a}}$ and $H^{\alpha}_{M,\bar{a}}$ is bounded by $E\bar{a}$.
\end{lemma}

\begin{proof} Immediate from the definitions in the $\alpha=0$ and successor $\alpha$ cases, and by induction for the limit $\alpha$ cases.
\end{proof}

  \begin{lemma} If $p\in \Omega$ and $J_{\bar{a}}$ is a function of any of the three types $F^{\alpha}_{M,\bar{a}}$, $G^{\alpha}_{M,N,\bar{a}}$ and
  $H^{\alpha}_{M,\bar{a}}$ then $\dom(J_{\bar{a}\on p}) \subseteq \dom(J_{\bar{a}})$ and if $x\in \dom(J_{\bar{a}\on p})$ then
    $J_{\bar{a}\on p}(x) = J_{\bar{a}}(x)$. \end{lemma}

  \begin{proof}
    The proof is by induction on $\alpha$. The salient cases are for
    $\alpha+1$ and arguments of the functions whose last element is
    $\alpha$, the other cases are all immediate from the definition or
    are direct consequences of the inductive hypothesis.

    It is clear that $\dom(J_{\bar{a}\on p}) \subseteq \dom(J_{\bar{a}})$ in each case since if
    $K\in \pshol$ and
    $\bar{t}\in {^{<\mu}|K|}$ is such that $E\bar{t} \le E\bar{a}\on p$ then $E\bar{t}\le E\bar{a}$.

    Now suppose $K\in \pshol$ and $\bar{t}\in {^{<\mu}|K|}$ is such that $E\bar{t} \le E\bar{a}\on p$. For $\bar{c}\in {^{\lh(\bar{t})}|M|}$ with $E\bar{c} \le E\bar{t}$
    we have $H^\alpha_{M,\bar{a}\bar{c}} = H^\alpha_{M,\bar{a}\on p \bar{c}}$, by inductive hypothesis. Thus
    $H^\alpha_{K,\bar{s}\bar{t}\on E\bar{c}} = H^\alpha_{M,\bar{a}\bar{c}}$ if and only if $H^\alpha_{K,\bar{s}\bar{t}\on E\bar{c}} = H^\alpha_{M,\bar{a}\on p \bar{c}}$.
    So $\bar{c}$ contributes $E\bar{c}$ to the supremum in the definition of
    $H^{\alpha+1}_{M,\bar{a}}(K,\bar{s}\bar{t},\alpha)$ if and only if $\bar{c}$ contributes $E\bar{c}$ to the
    supreme in the definition of $H^{\alpha+1}_{M,\bar{a}\on p}(K,\bar{s}\bar{t},\alpha)$. Hence the two suprema are equal.

    The proof is identical for $F^{\alpha+1}_{M,\bar{a}}$ and $G^{\alpha+1}_{M,N,\bar{a}}$.
    \end{proof}

\begin{lemma}\label{reln_Fs_and_Hs_tuples}   Let $M$, $\bar{a}\in {^{\lh(\bar{a})}|M|}$,  and $\alpha\in\On$.
Let $K\in\pshol$ and $\bar{s}\in {^{\lh(\bar{a})}|K|}$, $\bar{t}\in {^{<\mu}|K|}$ with $E\bar{t} \le E\bar{a}\land E\bar{s}$. Let $\beta<\alpha$.
  Then
                \[ F^{\alpha}_{M,\bar{a}}(F^\beta_{K,\bar{s}\bar{t}},\beta) = H^{\alpha}_{M,\bar{a}}(K,\bar{s}\bar{t},\beta) .\]
  Moreover, if further $N\in \pshol$ and $\bar{b}\in {^{\lh(\bar{a})}|N|}$ and we also have $E\bar{t}\le E\bar{b}$ then
  \[ F^{\alpha}_{M,\bar{a}} = F^{\alpha}_{N,\bar{b}} \hbox{ if and only if } H^{\alpha}_{M,\bar{a}} = H^{\alpha}_{N,\bar{b}} .\]
\end{lemma}

\begin{proof} We work by induction on $\alpha$. The limit $\alpha$ case is immediate by induction. For successor $\alpha$ with $\beta+1<\alpha$ the result is also immediate
  by induction. 
  For sucessor $\alpha+1$ we have
  \begin{eqn}
    \begin{split}
      F^{\alpha+1}_{M,\bar{a}}(F^\alpha_{K,\bar{s}\bar{t}},\alpha) =
           & \bigvee \setof{E\bar{c}}{\bar{c}\in |M| \sss\sss\&\sss\sss E\bar{c} \le E\bar{t} \sss\sss\& \sss\sss F^{\alpha}_{K,\bar{s}\bar{t}\on E\bar{c}} = F^{\alpha}_{M,\bar{a}\bar{c}}} \\
= & \bigvee \setof{E\bar{c}}{\bar{c}\in |M| \sss\sss\&\sss\sss E\bar{c} \le E\bar{t} \sss\sss\& \sss\sss H^{\alpha}_{K,\bar{s}\bar{t}\on E\bar{c}} = H^{\alpha}_{M,\bar{a}\bar{c}}} \\
= & H^{\alpha+1}_{M,\bar{a}}(K,\bar{s}\bar{t},\alpha)  .
    \end{split}
  \end{eqn}where the middle equality follows from by induction from the ``Moreover'' part of the conclusion for $\alpha$.
  However, the conclusion of the ``Moreover'' part of the conclusion
  for $\alpha+1$ is now immediate.  
\end{proof}

We next prove a lemma, for each of the functions, on applying the function to domain elements, introduced at successor stages, given by the information
in the subscript and one additional element of $|M|$.

\begin{lemma}\label{self_application_tuples} Let $M$, $N\in \pshol$, $\bar{a}\in {^{\lh(\bar{a})}|M|}$ and $\alpha\in\On$. Let $\bar{c}\in {^{<\mu}|M|}$ with $E\bar{c}\le E\bar{a}$.
  \[ F^{\alpha+1}_{M,\bar{a}}( F^{\alpha}_{M,\bar{a}\bar{c}},\alpha) = G^{\alpha+1}_{M,N,\bar{a}}( M,\bar{a}\bar{c},\alpha) = H^{\alpha+1}_{M,\bar{a}}( M,\bar{a}\bar{c},\alpha) =E\bar{c}.\]
\end{lemma}

\begin{proof} For $\bar{z}\in {^{<\mu}|M|}$ with $E\bar{z}\le E\bar{a}$ let $J_{\bar{z}}$ be either
  $ F^{\alpha}_{M,\bar{a}\bar{z}} $ or $ G^{\alpha}_{M,N,\bar{a}\bar{z}}$ or $H^{\alpha}_{M,\bar{a}\bar{z}}$.

  Then each of the first three expressions in the equation is by definition
  \[ \bigvee \setof{E\bar{e}}{\bar{e}\in |M| \sss\sss\&\sss\sss {E}\bar{e}
    \le E\bar{c} \sss\sss\& \sss\sss J_{\bar{c}\on {E}\bar{e}} = J_{\bar{e}}} .\]

  In each case, on the one hand $\bar{c}$ is an element of the set of $\bar{e}\in {^{<\mu}|M|} $ over which we take the supremum
  (since $\bar{c}\in |M|$, $E\bar{c}\le E\bar{c}$ and $J_{\bar{c}\on E\bar{c}} = J_{\bar{c}}$ -- the latter by
  Lemma (\ref{fn_range_bounded_tuples}) ).
  Thus the supremum is greater than or equal to $E\bar{c}$. On the other hand, for each $\bar{e}$ in the set we have
  $E\bar{e}\le E\bar{c}$ and thus the supremum is less than or equal to $E\bar{c}$. Hence, equality holds.
\end{proof}

In the following two propostions we need the definition of the $Q_\alpha(p)$ as given in the previous section.
  
\begin{proposition}\label{equiv_fn_equal_and_Q_for_Gs_tuples} Suppose $M$, $N\in \pshol$, $\bar{a}\in {^{\lh(\bar{a})}|M|}$, $\bar{b}\in {^{\lh(\bar{a})}|N|}$ and $E\bar{a} = E\bar{b}$. Then
  $G^\alpha_{M,N,\bar{a}} = G^\alpha_{N,M,\bar{b}}$ if and only if
  there is a partial isomorphism $h:\bar{a}\longrightarrow \bar{b}$ with $h\in Q_\alpha(E\bar{a})$.
\end{proposition}


\begin{remark} We could drop the assumption $E\bar{a}=E\bar{b}$ here
  and simply replace $\bar{a}$, $\bar{b}$ by $\bar{a}\on E\bar{a}\land E\bar{b}$ and
  $\bar{b}\on E\bar{a}\land E\bar{b}$ throughout the conclusion. 
\end{remark}

\begin{proposition}\label{equiv_fn_equal_and_Q_for_Fs_tuples}  Let $M\in \pshol$, $\bar{a}\in {^{\lh(\bar{a})}|M|}$ and $\alpha\in \On$.
Then for any $N\in \pshol$ and $\bar{b}\in {^{\lh(\bar{a})}|N|}$, and setting $p=E\bar{a}\land E\bar{b}$,
$F^\alpha_{M,\bar{a}\on p} = F^\alpha_{N,\bar{b}\on p}$ if and only if there is some $h:\bar{a}\longrightarrow \bar{b}$ with $h\in Q_\alpha(p)$.
\end{proposition}

\begin{corollary}\label{equiv_fn_eq_and_Q_for_Hs_tuples} By Lemma (\ref{reln_Fs_and_Hs_tuples}), the same proposition holds with $F^\alpha_{M,\bar{a}\on p} = F^\alpha_{N,\bar{b}\on p}$ replaced by
$H^\alpha_{M,\bar{a}\on p} = H^\alpha_{N,\bar{b}\on p}$ 
\end{corollary}

The ``$\Longrightarrow$'' proofs and the $\alpha=0$ and limit $\alpha$ cases of the ``$\Longleftarrow$'' proofs are the same for both proposition.
Consequently we give a unified treatment, distinguishing in the sucessor cases of the ``$\Longleftarrow$'' proofs.

\begin{proof} $\alpha=0$. Immediate from the definition of partial isomorphism (the $Q_0(p)$) and the functions $I_{M,\bar{a}}$ and $I_{N,\bar{b}}$.
  
  Limit $\alpha$. Immediate by induction and the definitions of the $Q_{\alpha}(p)$ and the $G^{\alpha}_{M,N,\bar{a}}$ and $G^\alpha_{N,M,\bar{b}}$,
  resp.~the $F^{\alpha}_{M,\bar{a}}$ and $F^\alpha_{N,\bar{b}}$.

    $\alpha+1$. ``$\Longrightarrow$''. By Lemma (\ref{self_application_tuples}), for any $\bar{c}\in {^{<\mu}|M|}$ with $E\bar{c}\le E\bar{a}$ we have
  $G^{\alpha+1}_{M,N,\bar{a}}(M,\bar{a}\bar{c},\alpha) = E\bar{c}$, resp.~$F^{\alpha+1}_{M,\bar{a}}(F^{\alpha}_{M,\bar{a}\bar{c}},\alpha) = E\bar{c}$.
  Hence, by the hypothesis that $G^{\alpha+1}_{M,N,\bar{a}} = G^{\alpha+1}_{N,M,\bar{b}}$,
  resp.~$F^{\alpha+1}_{M,\bar{a}\on p} = F^{\alpha+1}_{N,\bar{b}\on p}$,
  we have
  $G^{\alpha+1}_{N,M,\bar{b}}(M,\bar{a}\bar{c},\alpha) = E\bar{c}$, resp.~$F^{\alpha+1}_{N,\bar{b}}(F^{\alpha}_{M,\bar{a}\bar{c}},\alpha) = E\bar{c}$.
  By the definition of $G^{\alpha+1}_{N,M,\bar{b}}$, resp.~$F^{\alpha+1}_{N,\bar{b}}$, this gives
  \[ E\bar{c} = \bigvee\setof{E\bar{d}}{\bar{d}\in {^{\lh(\bar{c})}|N|} \sss\sss\&\sss\sss E\bar{d}\le E\bar{c}\sss\sss\& \sss\sss G^{\alpha}_{M,N,\bar{a}\bar{c}\on Ed} = G^\alpha_{N,M,\bar{b}\bar{d}}},\hbox{ resp.~}\]
  \[ E\bar{c} = \bigvee\setof{E\bar{d}}{\bar{d}\in {^{\lh(\bar{c})}|N|} \sss\sss\&\sss\sss E\bar{d}\le Ec\sss\sss\& \sss\sss F^{\alpha}_{M,\bar{a}\bar{c}\on E\bar{d}} = F^\alpha_{N,\bar{b}\bar{d}}}.\]
    So to see the ``forth'' property for $\bar{c}$ simply take $h_d(\bar{c}\on E\bar{d}) = \bar{d}$ for each $\bar{d}\in {^{\lh(\bar{c})}|N|}$ with $E\bar{d}\le E\bar{c}$ and
    $G^{\alpha}_{M,N,\bar{a}\bar{c}\on E\bar{d}} = G^\alpha_{N,M,\bar{b}\bar{d}}$, resp.~$F^{\alpha}_{M,\bar{a}\bar{c}\on E\bar{d}} = F^\alpha_{N,\bar{b}\bar{d}}$.
      By induction (\emph{i.e.}, using the result for $\alpha$), $h_{\bar{d}}\in Q_\alpha(E\bar{d})$, in each case, as required.      

      The ``back'' case, starting with any $\bar{d}\in {^{<\mu}|N|}$ with $E\bar{d}\le E\bar{b}$, is symmetrical.

      ``$\Longleftarrow$''.       Proof for Proposition (\ref{equiv_fn_equal_and_Q_for_Gs_tuples}).
     Observe, firstly, as $h\in Q_\alpha(\bar{a})$, by induction we have $G^\alpha_{M,N,\bar{a}} = G^\alpha_{N,M,\bar{b}}$.

     Secondly, by Lemma (\ref{self_application_tuples}), we know for all $\bar{c}\in {^{<\mu}|M|}$ with
      $E\bar{c}\le E\bar{a}$ and all $\bar{d}\in {^{\lh(\bar{c})}|N|}$ with $E\bar{d}\le E\bar{b} (=E\bar{a})$ we have
      $G^{\alpha+1}_{M,N,\bar{a}}(M,\bar{a}\bar{c},\alpha) = E\bar{c}$ and $G^{\alpha+1}_{N,M,\bar{b}}(N,\bar{b}\bar{d},\alpha) = E\bar{d}$.

        We also have, by the definition of $G^{\alpha+1}_{N,M,\bar{b}}$,
        \[ G^{\alpha+1}_{N,M,\bar{b}}(M,\bar{a}\bar{c},\alpha) \kern-.8pt  = \kern-.8pt
        \bigvee \setof{E\bar{d}}{\bar{d}\in {^{\lh(\bar{c})}|N|} \sss\sss\&\sss\sss
          E\bar{d}\le E\bar{c} \sss\sss\&\sss\sss G^{\alpha}_{M,N,\bar{a}\bar{c}\on E\bar{d}} = G^\alpha_{N,M,\bar{b}\bar{d}}}.\]

          By the hypothesis of ``$\Longleftarrow$,'' there is some set of triples
          $\setof{(h_i,q_i,\bar{d}_i)}{i\in I}$ with $\bar{d_i}\in {^{\lh(\bar{c})}|N|}$, $h_i\in Q_\alpha(E\bar{d}_i)$ extends $h_{q_i}$,
            $h_i``(\bar{c}\on Ed_i) = \bar{d}_i$ and ${\bigvee\setof{E\bar{d}_i}{i\in I} = E\bar{c}}$.

          By the inductive hypothesis we have for all $i\in I$ that $ G^\alpha_{M,N,\bar{a}\bar{c}\on E\bar{d}_i} = G^\alpha_{N,M,\bar{b}\bar{d}_i}$.
          So for each $i\in I$ we have that
          $\bar{d}_i$ is an element of the indexing set for the supremum in display two paragraphs above. Hence
          $G^{\alpha+1}_{N,M,\bar{b}}(M,\bar{a}\bar{c},\alpha) \ge E\bar{c}$.

          As for each $\bar{d}$ in the indexing set for that supremum we have $E\bar{d}\le E\bar{c}$, we also have $G^{\alpha+1}_{N,M,\bar{b}}(M,\bar{a}\bar{c},\alpha) \le E\bar{c}$.

          Hence $G^{\alpha+1}_{N,M,\bar{b}}(M,\bar{a}\bar{c},\alpha) = E\bar{c} = G^{\alpha+1}_{M,N,\bar{a}}(M,\bar{a}\bar{c},\alpha)$, as required.

          The proof that $G^{\alpha+1}_{N,M,\bar{b}}(N,\bar{b}\bar{d},\alpha) = E\bar{d} = G^{\alpha+1}_{M,N,\bar{a}}(N,\bar{b}\bar{d},\alpha)$ is identical
          using the ``back'' property of $h$  rather than the ``forth'' one.

``$\Longleftarrow$''. Proof for Proposition (\ref{equiv_fn_equal_and_Q_for_Fs_tuples}).

           We have
           \[ F^{\alpha+1}_{M,\bar{a}}(F^\alpha_{K,\bar{s}\bar{t}},\alpha) = \bigvee \setof{E\bar{c}}{\bar{c}\in {^{\lh(\bar{t})}|M|} \sss\sss\&\sss\sss E\bar{c} \le E\bar{t}
             \sss\sss\& \sss\sss F^{\alpha}_{K,\bar{s}\bar{t}\on E\bar{c}} = F^{\alpha}_{M,\bar{a}\bar{c}} } \]
  and
  \[ F^{\alpha+1}_{N,\bar{b}}(F^\alpha_{K,\bar{s}\bar{t}},\alpha) = \bigvee \setof{E\bar{d}}{\bar{d}\in {^{\lh(\bar{t}} |N|} \sss\sss\&\sss\sss E\bar{d} \le E\bar{t}
    \sss\sss\& \sss\sss F^{\alpha}_{K,\bar{s}\bar{t}\on E\bar{d}} = F^{\alpha}_{N,\bar{b}\bar{d}}} .\]

    By the ``forth'' property for $h$, for each $\bar{c}\in {^{\lh(\bar{t})}|M|}$ with $E\bar{c}\le E\bar{t}$ and $F^\alpha_{K,\bar{s}\bar{t}\on E\bar{c}} = F^\alpha_{M,\bar{a}\bar{c}}$ we have
          $\setof{(h^{\bar{c}}_i,q^{\bar{c}}_i,\bar{d}^{\bar{c}}_i)}{i\in I^{\bar{c}}}$ with $h^{\bar{c}}_i\in Q_\alpha(E\bar{d}^{\bar{c}}_i)$,  
    $h^{\bar{c}}_i``\bar{c}\on E\bar{d}^{\bar{c}}_i) = \bar{d}^{\bar{c}}_i$ and $\bigvee\setof{E\bar{d}^{\bar{c}}_i}{i \in I} = E\bar{c}$. By induction we thus have
        for each such $\bar{c}$ and $\bar{d}^{\bar{c}}_i$ that $F^\alpha_{M,\bar{a}\bar{c}\on E\bar{d}^{\bar{c}}_i} = F^\alpha_{N,\bar{b}\bar{d}^{\bar{c}}_i}$.

        So \begin{eqn}\begin{split}
            F^{\alpha+1}_{M,\bar{a}}(F^\alpha_{K,\bar{s}\bar{t}},\alpha) = \bigvee \setof{E\bar{d}^{\bar{c}}_i}{\exists \bar{c}\in {^{\lh(\bar{t})}|M|} \sss\sss\&\sss\sss & E\bar{d}^{\bar{c}}_i \le E\bar{c} \le E\bar{t}
              \sss\sss\& \sss\sss \\
              & F^{\alpha}_{K,\bar{s}\bar{t}\on E\bar{d}^{\bar{c}}_i}  = F^{\alpha}_{M,\bar{a}\bar{c}\on E\bar{d}^{\bar{c}}_i} = F^\alpha_{N,\bar{b}\bar{d}^{\bar{c}}_i} }.
          \end{split}
        \end{eqn}

          So $F^{\alpha+1}_{M,\bar{a}}(F^\alpha_{K,\bar{s}\bar{t}},\alpha)  \le F^{\alpha+1}_{N,\bar{b}}(F^\alpha_{K,\bar{s}\bar{t}},\alpha)$.

              However, a symmetrical argument using the ``back'' property of $h$ shows the reverse inequality. Thus we have equality as desired.                
\end{proof}

We now define a `Scott' rank for sheaves of models. The treatment is similar to the classical one, as given for example in \cite{Marker}.

\begin{proposition} Let $M\in \pshol$. There is an ordinal $\alpha$ such that if $\bar{a}$, $\bar{b}\in {^n|M|}$ with $E\bar{a}=E\bar{b}$ and
  $(M,\bar{a})\sim_\alpha^{E\bar{a}} (M,\bar{b})$ then for all
  $\beta\in \On$ we have $(M,\bar{a})\sim_\beta^{E\bar{a}} (M,\bar{b})$. The least such $\alpha$ is the \emph{Scott rank} of $M$.
\end{proposition}

\begin{proof}  Let $\Gamma_\alpha = \bigcup_{n<\omega}\setof{(\bar{a},\bar{b})}{\bar{a},\sss\bar{b}\in   {^n|M|} \sss\sss\&\sss\sss (M,\bar{a})\not\sim_\alpha^{E\bar{a}} (M,\bar{b})}$.
  Clearly, if $\alpha<\beta$ then
    $\Gamma_\alpha \subseteq \Gamma_\beta$. 

  First of all, we show by induction that if $\Gamma_\alpha = \Gamma_{\alpha +1}$ then for all $\beta>\alpha$ we have $\Gamma_\alpha = \Gamma_\beta$. For $\beta=\alpha+1$ this is trivial. If $\beta$ is a limit and
  $\Gamma_\alpha = \Gamma_\gamma$ for all $\gamma\in (\alpha,\beta)$ then by the definition of $\sim_\beta^p$ we have that $\Gamma_\alpha = \Gamma_\beta$.

  Finally, suppose $\beta>\alpha$ and
  $\Gamma_\alpha = \Gamma_\beta$, $(M,\bar{a})\sim^{E\bar{a}}_\beta (M,\bar{b})$. Suppose $\bar{c}\in {^{<\mu}|M|}$. Since $(M,\bar{a})\sim^{E\bar{a}}_{\alpha+1} (M,\bar{b})$ there is
  some $h:M\longrightarrow M\in Q_{\alpha+1}(E\bar{a})$ with $h:\bar{a}\longrightarrow \bar{b}$,  and, hence, by the definition of $Q_{\alpha+1}(E\bar{a})$,  $h\in Q_{\beta}(E\bar{a})$ and
  there is some $\setof{(h_j,q_j,\bar{d}_j)}{j\in J}$ such that for each $j\in J$ we have $h_j\in Q_\beta(q_j)$, $h_{q_j} \subseteq h_j$, $ \bar{c}\on q_j \in \dom (h_j)$, $h_j``\bar{c}\on q_j=\bar{d}_j$,
  $E\bar{d}_j=q_j$
  and $\Vee_{j\in J} q_j = E\bar{c}$.

  For each $j\in J$ we thus have that $(M,\bar{a}\bar{c}\on q_j)\sim^{q_j}_\alpha (M,\bar{b}\bar{d}_j)$. By the inductive assumption, $(M,\bar{a}\bar{c}\on q_j)\sim^{q_j}_\beta (M,\bar{b}\bar{d}_j)$.
  Thus $(M,\bar{a})\sim^{E\bar{a}}_{\beta+1} (M,\bar{b})$.

  The case where we have $\bar{d}\in {^{<\mu} |M|}$ and find $\bar{c}_i\in {^{\lh(\bar{d})}|M|}$ using the ``back'' property is similar.

  Clearly, for each $\alpha$ we have that $\Gamma_\alpha\subseteq {^n|M|}\times {^n|M|}$, and so $\card{\Gamma_\alpha} < \card{M}^+$ and the sequence of $\Gamma_\alpha$ must have length less than
  $\alpha+1$.
\end{proof}

\begin{remark}
If $\Omega = \set{\bot,\top}$, \emph{i.e.}, if we are in classical model theory, we could now prove a version of Scott's isomorphism theorem.
  
\begin{proposition}\label{Scott_iso_thm} (\cite{BMPIV}, Proposition (3.9) Suppose $\Omega = \set{\bot,\top}$.
  Let $M$, $N$ be countable $L$-structures and $\alpha$ the Scott rank of $M$. Then
  $N\isom M$ if and only if $H^\alpha_{M,\emptyset} = H^\alpha_{N,\emptyset}$ and for all $\bar{a}\in {^{<\lambda}M}$ and $\bar{b}\in {^{\lh(\bar{a})}N}$
  if $H^\alpha_{M,\bar{a}} = H^\alpha_{N,\bar{b}}$ then $H^{\alpha+1}_{M,\bar{a}} = H^{\alpha+1}_{N,\bar{b}}$.
\end{proposition}

However, for arbitrary complete Heyting algebras $\Omega$ the conditions on the $H^\alpha$ are not so helpful
in constructing a sequence of partial isomorphims forming a chain under inclusion of which one could take the union after $\omega$-many steps.
Rather one would obtain a system of functions with some increasing coherence, but on
smaller and smaller portions of $|M|$. It is not clear to us, at this point, how one might trade off
such a system of functions for a global similarity property (such as isomorphim is) going beyond what
can be translated directly from the back and forth properties themselves.
\end{remark}

\section{Languages}\label{languages}

An explicit account of the construction of a first order language over
a signature $\tau$ can be found, for example in \cite{Rothmaler}.
We adopt (a mild adaptation of) V\"a\"an\"anen's definition  of
the infinitary language $L_{\kappa\lambda}$ (\cite{Vaananen}), Definition (9.12)
for a concretete definition of infinitary languages.


\begin{definition}\label{defn_L_kappa_lambda;mu}
  Let $\mu$, $\lambda$ and $\kappa$ be cardinals.
  The set $L_{\kappa\lambda;mu}$ is defined by
  \begin{enumerate}
  \item For $\alpha<\lambda$ the variable $v_\alpha$  is a term. Each constant symbol $c\in  \mathcal C$ is a term.
    If $f$ is an $n$-ary
    function symbol and $t_0$, \dots, $t_{n-1}$ are terms then so is $f(t_0,\dots,t_{n-1})$.
\item If $t_0$ and $t_1$ are terms, then $t_0=t_1$ is an $L_{\kappa\lambda;\mu}$-formula.
\item If $t_0$,\dots, $t_{n-1}$ are terms and $R$ is an $n$-ary relation symbol then\break
   $R(t_0,\dots t_{n-1})$ is an $L_{\kappa\lambda;\mu}$-formula.
\item If $\phi$ is an $L_{\kappa\lambda;\mu}$-formula, so is $\neg \phi$.
\item If $\Phi$ is a set of $L_{\kappa\lambda;\mu}$-formulae of size less than $\kappa$
  with a fixed set $V$ of free variables and $\card{V} < \lambda$,
 $\Wedge\Phi$ is an $L_{\kappa\lambda;\mu}$-formula.
\item If $\Phi$ is a set of $L_{\kappa\lambda;\mu}$-formulae of size less than $\kappa$
  with a fixed set $V$ of free variables and $\card{V} < \lambda$,
  $\Vee\Phi$ is an $L_{\kappa\lambda;\mu}$-formula.
\item If $\phi$ is an $L_{\kappa\lambda;\mu}$-formula,
  $V \in [\setof{v_\alpha }{\alpha < \lambda}]^{<\mu}$ and $\card{V}<\lambda$,
   $\forall V\sss\phi$ is an $L_{\kappa\lambda;\mu}$-formula.
\item If $\phi$ is an $L_{\kappa\lambda;\mu}$-formula,
  $V \in [\setof{v_\alpha }{\alpha < \lambda}]^{<\mu}$ and $\card{V}<\lambda$,
  $\exists V\sss\phi$ is an $L_{\kappa\lambda;\mu}$-formula.
  \end{enumerate}  
\end{definition}

Thus $\lambda$ bounds the size of the set of free variables of a formula,
$\kappa$ bounds the size of the infintary conjuctions and disjunctions, and
$\mu$ bounds the size of the sets of variables over which one may quantify.

\begin{definition}\label{defn_L_infty_lambda} For $\mu$, $\lambda$ and $\kappa$
  cardinals we set $L_{\kappa\lambda} = L_{\kappa\lambda;\lambda}$
  and $\bar{L}_{\kappa\lambda} = L_{\kappa\lambda;2}$.   We also set
  $L_{\infty\lambda} = \bigcup_{\kappa\in \Card} L_{\kappa\lambda}$
  and
  $\bar{L}_{\infty\lambda} = \bigcup_{\kappa\in \Card}
  \bar{L}_{\kappa\lambda}$.
\end{definition}

So, for example, $\bar{L}_{\omega\omega}$ is first order logic over $\tau$.

In this paper we work with languages over the signature $\tau$ in a way that
allows us to treat classical first order logic ($\bar{L}_{\omega\omega}$) and the
infinitary logics $\bar{L}_{\infty\omega}$ and $L_{\infty\lambda}$, for $\lambda$ a cardinal,
simultaneously.

Thus we consider only languages, $L$, with $L=\bar{L}_{\omega\omega}$, $L=\bar{L}_{\infty\omega}$ or
$L=L_{\infty\lambda}$.


\begin{definition} Let $\phi \in L$. We say $\phi$ is \emph{unnested} if
  every atomic subformula of $\phi$ is unnested.
  Write ${\mathcal U}$ for the collection of unnested formulae.
\end{definition}

\begin{remark} If $L$ is relational every formula is unnested. 
\end{remark}

\begin{definition}\label{defn_pp_fmla} A formula $\phi \in L$ is
  a \emph{positive primitive} formula~(\emph{pp formula}), if it is made up from unnested atomic formulae
  using only conjunction and existential quantification.

\end{definition}

The next lemma enables us to show there is a normal form,
at least in terms of evaluated extent, for positive primitive
formulae in $\bar{L}_{\omega\omega}$.

\begin{lemma} Suppose $\phi_0(\bar{v},\bar{u})$ and $\phi_1(\bar{v},\bar{u})$ are $L$-formulae with
  $\lh(\bar{v})+\lh(\bar{u})$-free variables,
  $M\in \pshol$ and  $\bar{a}\in {^{\lh(\bar{v})} |M|}$. Then
  \[ [\exists \bar{v}_0 \sss \phi_0(\bar{a},\bar{v}_0) \land \exists \bar{v}_1 \sss \phi_1(\bar{a},\bar{v}_1)]_M =
  [\exists \bar{v}_0 \sss \exists \bar{v}_1 \sss ( \phi_0(\bar{a},\bar{v}_0) \land  \phi_1(\bar{a}),\bar{v}_1)]_M .\]
 \end{lemma}

\begin{proof} Start by noting that
  $E\exists \bar{v}_0 \sss \phi_0(\bar{v},\bar{v}_0) = E\phi_0(\bar{v},\bar{v}_0)$, $= E\phi_0$, say, where
  $E\phi_0 = \Wedge\setof{E^Cc}{c\hbox{ occurs in }\phi_0}$, and
  similarly for $\phi_1$. Let $q = E\phi_0\land E\phi_1$.

  By the definition of interpretation of formulae (\cite{BMPI}, Definition (2.20))
  and the distributivity of $\land$ over $\Vee$ we have
  \begin{eqn}
    \begin{split}
      [\exists \bar{v}_0 \sss \phi_0(\bar{a},\bar{v}_0) \land & \exists \bar{v}_1 \sss \phi_1(\bar{a},\bar{v}_1)]_M \\
      & = \sss\sss q \land E\bar{a} \land
      [\exists \bar{v}_0 \sss \phi_0(\bar{a},\bar{v}_0)]_M  \land [ \exists \bar{v}_1 \sss \phi_1(\bar{a},\bar{v}_1)]_M \\
      & = \sss\sss q \land E\bar{a} \land
      (\Vee_{\bar{s}\in {^{\lh(\bar{u})} |M|}}  [\phi_0(\bar{a},\bar{s})]_M ) \land
            ( \Vee_{\bar{t}\in {^{\lh(\bar{u})} |M|}}  [\phi_1(\bar{a},\bar{t})]_M)  \\
      & = \sss\sss q \land E\bar{a} \land
      \Vee_{\bar{s}\in {^{\lh(\bar{u})} |M|}} (  \Vee_{\bar{t}\in {^{\lh(\bar{u})} |M|}} ( [\phi_0(\bar{a},\bar{s})]_M
             \land [\phi_1(\bar{a},\bar{t})]_M   ) )\\
      & = \sss\sss q \land E\bar{a} \land
             \Vee_{\bar{s}\in {^{\lh(\bar{u})} |M|}} (  \Vee_{\bar{t}\in {^{\lh(\bar{u})} |M|}}
                  ( [\phi_0(\bar{a},\bar{s})  \land \phi_1(\bar{a},\bar{t})]_M ) )\\
      & = \sss\sss q \land E\bar{a} \land
      \Vee_{\bar{s}\in {^{\lh(\bar{u})} |M|}} (  [ \exists \bar{v}_1 ( \phi_0(\bar{a},\bar{s})  \land [\phi_1(\bar{a},\bar{v}_1)]_M ) )\\
        & =  [\exists \bar{v}_0 \sss (  \exists \bar{v}_1\ssss ( \phi_0(\bar{a},\bar{v}_0)
             \land \phi_1(\bar{a},\bar{v}_1) )) ]_M  .
    \end{split}
\end{eqn}
\end{proof}

\begin{corollary}\label{normal_form_for_pp} For any positive primitive formula $\phi(\bar{v})$ in $\bar{L}_{\omega\omega}$
  there is a collection $\setof{\phi_i}{i<m}$ of
  unnested atomic formulae and a string of existential quantifiers
  $\exists v_0\sss\dots\sss\exists v_{n-1}$ such that setting $\phi^*$
  to be $\exists v_0\sss\dots\sss\exists v_{n-1} (\phi_0 \land \dots
  \land \phi_{k-1})$, for all $\bar{a}\in {^{\lh(\bar{u})} |M|}$ we have
  $[\phi(\bar{a})]_M = [\phi^*(\bar{a})]_M $.
\end{corollary}

We do not have similar results for languages in which infinitary conjunctions and disjunctions are allowed
as we do not in general have distributivity of $\Wedge$ over $\Vee$ in complete Heyting algebras.

We introduce the notion of \emph{quantifier degree} which we will use in the ensuing.

\begin{definition}\label{quantifier_degree} (Quantifier degree) Let $\phi \in L_{\infty\lambda}$.
  The \emph{quantifier degree} of $\phi$, $d(\phi)$, is defined by induction on its complexity.
  \begin{center}
\vskip-10pt
\begin{tabular}{ll}
 $\phi$ atomic  & \qquad $d(\phi)= 0$ \\
  $\phi = \neg \psi$ & \qquad $d(\phi)= d(\psi)$ \\
  $\phi = \psi \to \chi$ &  \qquad $d(\phi)= \max(\set{d(\psi), d(\chi)})$ \\
  $\phi = \Wedge \Phi$ & \qquad $d(\phi) = \sup\setof{d(\psi)}{\psi \in \Phi}$ \\
  $\phi = \Vee \Phi$ & \qquad $d(\phi) = \sup\setof{d(\psi)}{\psi \in \Phi}$ \\ 
  $\phi = \forall V \psi$ & \qquad $d(\phi) = d(\psi) + 1$ \\
  $\phi = \exists V \psi$ & \qquad $d(\phi) = d(\psi) + 1$ \\
\end{tabular}
\end{center}
\end{definition} 

\begin{note-nono} When $\lambda=\omega$ this definition induces the `usual' quantifier degree on
  $\bar{L}_{\omega\omega}$ or $\bar{L}_{\infty\omega}$.
  Observe, for example, writing $\exists x $ for $\exists \set{x}$ for $\set{x}$ a singleton,
  if $\phi$ is $\exists \set{x,y}\sss x=y$ and $\psi$ is $\exists x
  \exists y  \sss x=y$ then $d(\phi)=1$, but $d(\psi)=2$.
\end{note-nono}

We also introduce a \emph{modified rank}, which takes into account the complexity gained in allowing function symbols in
the language. This rank is a slight variant of that found in
\cite{EFT}, Exercise (XII.3.15),  which in turn is a variant of
a rank found in \cite{Flum}, p.254. It is used in \cite{EFT} to give a framework for a proof of a Fra\"iss\'e-style theorem
for arbitrary finite first order languages (\emph{i.e.},
possibly with function symbols and symbols for constants).


\begin{definition}\label{modified_rank} (Modified Rank) Let $\phi \in L_{\infty\lambda}$.
  The \emph{modified rank} of terms $t$ and formulae $\phi$, $r(t)$, $r(\phi)$, is defined by induction on their complexity.
  \begin{center}
\vskip-10pt
\begin{tabular}{ll}
  $t=v$  & \qquad $r(\phi)= 0$ \\
  $t = c$ & \qquad $r(\phi)= 1$ \\
  $t= f(s_0,\dots,s_{n-1})$ & \qquad $r(\phi)= 1 + r(s_0) + \dots r(s_{n-1})$ \\
  $\phi = R(t_0,\dots,t_{m-1})$ & \qquad $r(\phi)= r(t_0) + \dots r(t_{m-1})$ \\
  $\phi = \hbox{`}t_0 = t_1\hbox{'}$ & \qquad $r(\phi)=  \max\set{0,r(t_0) + r(t_1) -1}$ \\
  $\phi = \neg \psi$ & \qquad $r(\phi)= r(\psi)$ \\
  $\phi = \psi \to \chi$ & \qquad $r(\phi)= \max(\set{r(\psi), r(\chi)})$ \\
  $\phi = \Wedge \Phi$ & \qquad $r(\phi) = \sup\setof{r(\psi)}{\psi \in \Phi}$ \\
  $\phi = \Vee \Phi$ & \qquad $r(\phi) = \sup\setof{r(\psi)}{\psi \in \Phi}$ \\ 
  $\phi = \forall V \psi$ & \qquad $r(\phi) = r(\psi) + 1$ \\
  $\phi = \exists V \psi$ & \qquad $r(\phi) = r(\psi) + 1$ \\
\end{tabular}
\end{center}
\end{definition} 

\begin{remark}\label{m_rank_le_qdeg} Note for any formula $\phi$ we have $d(\phi)\le r(\phi)$.
If the language is relational then we always have equality: $d$ and $r$ are the same.
However in languages without function symbols, but with constants we already have $r\ne d$. 
\end{remark}

\begin{definition} Let $\alpha$ be an ordinal.
  $L^{\alpha}$ is the collection of $L$-formulae $\phi$
  such that $r(\phi)\le \alpha$.
\end{definition}

\section{Unnesting formulae}\label{unnesting}

In this section we show that every formula is equivalent, in terms of
evaluated extent, to an unnested one. The analogous result for
classical model theory is relatively routine, although most model
theory textbooks appear not to touch on the question. An exception is
\cite{Hodges}, which gives a `proof by example' (Theorem (2.6.1)). In
the context of presheaves of models one has to work somewhat harder.

Recall that $L$-terms are defined by recursion, with base cases being
that of being either a variable symbol or a constant symbol and the
recursive mechanism being the application of an $n$-ary function
symbol to a sequence of $n$-many already defined $L$-terms.

\begin{definition}\label{defn_nested} A an atomic $L$-formula $\phi$ is \emph{nested}
  if it is of one of the following forms:
  \begin{itmz}
  \item $c= d$, where $c$, $d\in {\mathcal C}$ are constant symbols
  \item $f(t_0,\dots,t_{n-1}) =t_n$ or $t_n = f(t_0,\dots,t_{n-1})$, where $f$ is an $n$-ary function symbol for some $n<\omega$ and $t_0,\dots,t_n$ are $L$-terms, not all of which are variable symbols.
    \item $R(t_0,\dots,t_{n-1})$, where $R$ is an $n$-ary relation symbol for some $n<\omega$ and $t_0,\dots,t_{n-1}$ are $L$-terms, not all of which are variable symbols.
  \end{itmz}
\end{definition}

\begin{definition}\label{defn_phi^exists} We recursively define for each atomic $L_{\infty\lambda}$-formula $\phi(\bar{v})$,
  with free variables $\bar{v}$, an unnested, in fact positive primitive, formula $\phi^\exists(\bar{v})$.

For unnested atomic formulae we simply take $\phi^\exists = \phi$.

If $\phi$ is $c=d$, where $c$, $d$ are constant symbols we set
$\phi^\exists$ to be $\exists v_0 \sss v_0 = c\sss\sss\&\sss\sss v_0 = d $.

Now, suppose $\phi$ is $f(t_0,\dots,t_{n-1}) =t_n$, where $f$ is an
$n$-ary function symbol for some $n<\omega$ and $t_0,\dots,t_n$ are
$L$-terms, not all of which are variable symbols.

Suppose that $S = \setof{i_j}{j\le m}$ is an enumeration of the collection of $i\le n$ such that $t_i$ is not a variable symbol.
For $i\ni S$ let $v_i = t_i$ and let $\setof{v_{i_j}}{j\le m}$ be a collection of distinct variable symbols
none of which appear in any of the $t_i$ for $i<n$. We set $\phi'$ to be
\[\exists v_{i_0}\sss\dots\sss \exists v_{i_{m}} \sss\sss
v_{i_0} = t_{i_0} \sss\sss\&\sss\sss \dots \sss\sss\&\sss\sss v_{i_{m}} = t_{i_{m}} \sss\sss\&\sss\sss   f(v_0,\dots,v_{n-1}) = v_n.\] 
We then set $\phi''$ to be 
\begin{eqn}
  \begin{split} \exists v_{i_0}\sss\dots\sss \exists v_{i_{m}} \sss\sss (v_{i_0} = t_{i_0}& )^{\exists} \sss\sss\&\sss\sss \dots
    \sss\sss\&\sss\sss (v_{i_{m}} = t_{i_{m}})^\exists\\
    &
    \sss\sss\&\sss\sss f(v'_0,\dots,v'_{n-1}) = v'_n
  \end{split}
\end{eqn}
where for $i\ni S$ we have $v'_i = v_i$ and for $i=i_j\in S$ we have $v'_i = v_{i_j}$.

The definition for $\phi$ being $t_n = f(t_0,\dots,t_{n-1})$ is almost
identical, we just switch the order of the final conjunct to be $ v_n
= f(v_0,\dots,v_{n-1})$.

Finally, suppose $\phi$ is $R(t_0,\dots,t_{n-1})$, where $R$ is an
$n$-ary relation symbol for some $n<\omega$ and $t_0,\dots,t_{n-1}$
are $L$-terms, not all of which are variable symbols.  We have a
similar definition to the previous one.

Suppose that $S = \setof{i_j}{j\le m}$ is an enumeration of the
collection of $i< n$ such that $t_i$ is not a variable symbol.  For
$i\ni S$ let $v_i = t_i$ and let $\setof{v_{i_j}}{j\le m}$ be a
collection of distinct variable symbols none of which appear in any of
the $t_i$ for $i<n$.  We then set $\phi''$ to be
\begin{eqn}
  \begin{split} \exists v_{i_0}\sss\dots\sss \exists v_{i_{m}} \sss\sss (v_{i_0} = t_{i_0}& )^{\exists} \sss\sss\&\sss\sss \dots
    \sss\sss\&\sss\sss (v_{i_{m}} = t_{i_{m}})^\exists\\
    &
    \sss\sss\&\sss\sss R(v'_0,\dots,v'_{n-1}),
  \end{split}
\end{eqn}
where for $i\ni S$ we have $v'_i = v_i$ and for $i=i_j\in S$ we have $v'_i = v_{i_j}$.


\end{definition}

\begin{definition} If $\psi$ is an arbitrary $L$-formula let $\psi^\exists$ be the formula derived from $\psi$ by for each atomic subformula $\phi$ replacing $\phi$ by $\phi^\exists$.
\end{definition}

\begin{proposition}\label{rank=degree_of_ext} If $\psi$ is an arbitrary $L$-formula then $r(\psi)=d(\psi^\exists)$.
\end{proposition}

\begin{proof} Immediate from the definitions.
\end{proof}

We will show for $L$-formulae $\psi(\bar{v})$, $M\in \pshol$ and
$\bar{a}\in {^{\lh(\bar{v})}|M|}$ that $[\psi(\bar{a})]_M =
[\psi^\exists(\bar{a})]_M$. The crux of the proof will be to prove
this for atomic formulae. The easy remainder of the proof is via the
recursive definition of interpretation of formulae.

The following lemma from \cite{BMPI} will be useful.

\begin{lemma}\label{char_funct} (\cite{BMPI}, Lemma (2.25).) Let $\bar{v} = \tup{v_0, \ldots, v_{n-1}}$ be variables and
  let $t( \bar {v})$, $s( \bar {v})$  be $L$-terms. Let $M \in pSh(\Omega, L)$, with
  $\bar{a} = \tup{a_0, \ldots, a_{n-1}}$, $\bar{b}=\tup{ b_0, \ldots,  b_{n-1}} \in {^n|M \on Et \wedge Es| }$.
  Then
  \[ [\bar{a} = \bar{b}]_M \wedge [ s^M (\bar{a}) = t^M (\bar{a})]_M =
     [\bar{a} = \bar{b}]_M \wedge [ s^M (\bar{b}) = t^M (\bar{b})]_M .\]

     Particularly, $ [\bar a = \bar b]_M \wedge [ s^M (\bar{a}) = t^M (\bar{a})]_M \leq
                        [ s^M (\bar{b}) = t^M (\bar{b})]_M$.
\end{lemma}

\begin{theorem}\label{unnesting_pres_value_atom_fml}
Let $\phi(\bar{v})$ be an atomic formula with $\bar{v} = \tup{v_0,
  \dots, v_{n-1}}$. Let $M \in pSh(\Omega,L)$ and $\bar{a} =\tup{a_0,
  \dots, a_{n-1}} \in |M|$. Then $[\phi(\bar{a})]_M =
[\phi^\exists(\bar{a})]_M$.
\end{theorem}

\begin{proof}
It is immediate the theorem holds for the unnested atomic formulae since for these formulae $\phi^\exists =\phi$.. 
We  show remainder of the theorem by induction on complexity of $L$-terms appearing in atomic formulae.
Case 1 is a base case, along with the theorem for unnested formulae, cases 2 and 3 are induction steps. 

Case 1: Let $\phi(\bar v)$ be ``$c=d$'' with $c,d \in |C|$ constant symbols. Let $M \in pSh(\Omega,L)$ and $\bar{a} = \tup{a_0, \dots, a_{n-1}} \in {^n|M|}$.
By \cite{BMPI}, Definition (2.14), we have $E^Cc = E^M \mathfrak{c}^M$ and $E^Cd =E^M \mathfrak{d}^M$. Set $p= E^Cc \wedge E^C d$.

It will be useful to note $E(\mathfrak{c}^M\on E(\bar{a} \on p))  = E(\bar{a}\on p)$, which is so since
\[ E(\mathfrak{c}^M\on E(\bar{a} \on p)) = E\mathfrak{c}^M\land E\bar{a} \land p = E\bar{a} \land p = E(\bar{a}\on p), \]
where the first and third equalities come from applications of the interaction between extents and restrictions in presheaves (\cite{BMPI}, Definition (1.13)),
and the second equality comes from the fact that $p\le E\mathfrak{c}^M$ by the definition of $p$.

We have 
\begin{eqn}\begin{split} [\phi(\bar a)]_M = & \sss [(c=d)(\bar a)]_M = [c^M(\bar a \on p) = d^M(\bar a \on p)]_M  \\ = & \sss
    [\mathfrak{c}^M\on E(\bar a \on p) 
      =  \mathfrak{d}^M \on E(\bar a \on p)]_M ,
  \end{split}
\end{eqn}where the second equality comes from \cite{BMPI}, Definition (2.23), the third from  \cite{BMPI}, Definition (2.17).
We also have
\begin{eqn}\begin{split}  [\mathfrak{c}^M\on E(\bar a \on p) =  \mathfrak{d}^M \on E(\bar a \on p)]_M \le E(\mathfrak{c}^M \on E(\bar{a} \on p)) = E(\bar{a}\on p),
  \end{split}
\end{eqn}{}where the inequality comes from \cite{BMPI}, Lemma (1.16) and the equality is as noted above. 

Thus
\[ [\phi(\bar{a})]_M = [\mathfrak{c}^M\on E(\bar{a} \on p) =  \mathfrak{d}^M \on E(\bar{a} \on p)]_M \land E(\bar{a}\on p).\]

On the other hand, $ \phi^\exists (\bar v)$ is $ \exists x (x=c \wedge x=d).$ Thus \vskip6pt
\begin{eqn}\begin{split} [\phi^\exists (\bar{a} &)]_M  = [\exists x (x=c \wedge x=d) (\bar{a})]_M = \bigvee_{s\in |M|}  [ (s=c \land s=d)(\bar{a})]_M  \\
    & = \bigvee_{s\in |M|}  E\bar{a}\land p \wedge [ (s=c)(\bar{a})]_M \wedge [ (s=d)(\bar{a})]_M  \\
    & =  E(\bar{a}\on p) \wedge \bigvee_{s\in |M|} ( [ s \on p=\mathfrak{c}^M\on E(\bar{a} \on p) ]_M \wedge [ s \on p=\mathfrak{d}^M \on E(\bar{a} \on p)]_M ),
    \end{split}
\end{eqn}where the second and third equalities comes from two separate clauses of \cite{BMPI}, Definition (2.23), the fourth from \cite{BMPI}, Definition (2.17).

By \cite{BMPI}, Lemma (1.17), which follows from the definition of $[.=.]_M$, for each $s\in |M|$ we have 
\begin{eqn}\begin{split} [ (s \on p=\mathfrak{c}^M\on E(\bar{a} \on p) ]_M \wedge [ (s \on p=\mathfrak{d}^M & \on E(\bar{a} \on p)]_M \\
    & \le [\mathfrak{c}^M\on E(\bar{a} \on p) = \mathfrak{d}^M \on E(\bar{a} \on p)]_M. 
  \end{split}
\end{eqn}

Thus \[ [\phi^\exists (\bar{a})]_M \le E(\bar{a}\on p) \land [\mathfrak{c}^M\on E(\bar{a} \on p) = \mathfrak{d}^M \on E(\bar{a} \on p)]_M, \sss\sss = [\phi(\bar{a})]_M .\]

However, $\mathfrak{c}^M \on E(\bar a \on p) \in |M|$ and so is one of the sections $s$ over which the supremum to which we showed $ [\phi^\exists (\bar{a})]_M$ is equal is taken.

Thus \begin{eqn}\begin{split} [\phi^\exists (\bar{a})]_M & \ge  E(\bar{a}\on p) \land [ \mathfrak{c}^M\on E(\bar{a} \on p)  =\mathfrak{c}^M\on E(\bar{a} \on p) ]_M \\
    &\sss \qquad \qquad \qquad \wedge [ \mathfrak{c}^M\on E(\bar{a} \on p)   = \mathfrak{d}^M \on E(\bar{a} \on p)]_M \\
    &  = \sss E(\bar{a}\on p) \land E(\mathfrak{c}^M\on E(\bar{a} \on p)) \land [ \mathfrak{c}^M\on E(\bar{a} \on p)  \on p=\mathfrak{d}^M \on E(\bar{a} \on p)]_M \\
    &  = \sss E(\bar{a}\on p) \land [ \mathfrak{c}^M\on E(\bar{a} \on p)  \on p=\mathfrak{d}^M \on E(\bar{a} \on p)]_M = [\phi(\bar{a})]_M, 
  \end{split}
\end{eqn}where the penultimate equality holds since as noted at the start of the proof for this case,
$E(\mathfrak{c}^M\on E(\bar{a} \on p)) = E(\bar{a}\on p)$.

Case 2: Let $\phi(\bar{v})$ be ``$ t_m (\bar{v}) =f(t_0(\bar{v}), \dots, t_{m-1}(\bar{v}))$'' where for $j \le m$ we have $t_j (\bar{v})$ are $L$-terms which are
not all variable symbols and $f$ is an $m$-ary function symbol. Let 
$M \in pSh(\Omega,L)$ and $\bar{a}=\tup{a_0, \dots, a_{n-1}} \in {^n|M|}$. Let $p = E\phi = \bigwedge_{j \le m} Et_j$.

Firstly, by \cite{BMPI}, Definitions (2.23) and (2.17), 
\begin{eqn}\begin{split}
    [\phi(\bar{a})]_M =  & \sss [t_m^M (\bar{a} \on p) = f^M( t_0^M (\bar{a} \on p), \dots, t_{m-1}^M (\bar{a} \on p))]_M = \\ & \sss E \bar{a}\on p \wedge [t_m^M (\bar{a} \on p) = f^M( t_0^M (\bar{a} \on p), \dots, t_{m-1}^M (\bar{a} \on p))]_M. 
  \end{split}
\end{eqn}

On the other hand, by definition, we have  $\phi^\exists (\bar{v})$ is
\[ \exists x_0 \dots \exists x_m ((x_0= t_0(\bar{v}))^\exists \wedge \dots \wedge (x_m = t_m (\bar{v}))^\exists \wedge x_m=f(x_0, \dots x_{m-1})). \] 
By induction hypothesis, for each $j<m$, we have
\[ [(x_j = t_j (\bar{v}))^\exists (\bar{a})]_M = [x_j = t_j (\bar{v}) (\bar{a})]_M . \] 

Using  \cite{BMPI}, Definitions (2.23) and (2.17) and the induction hypothesis,
\begin{eqn}\begin{split}
    [\phi^\exists (\bar{a})]_M  = & \sss\bigvee_{s_0 \in |M|} \dots \bigvee_{s_m \in |M|} ( E\bar{a} \on p \wedge [s_0 \on p= t_0^M(\bar{a} \on p)] \wedge \dots  \\
    & \sss\qquad\wedge [s_m \on p = t_m^M (\bar{a} \on p)] \wedge [s_m \on p=f^M(s_0 \on p, \dots, s_{m-1} \on p]_M)   \\
    \le  & \sss  E\bar{a} \on p \wedge [t_m^M (\bar{a} \on p) =  f^M( t_0^M (\bar{a} \on p), \dots, t_{m-1}^M (\bar{a} \on p))]_M = [\phi(\bar{a})]_M,
  \end{split}
\end{eqn} where the inequality holds  by an application of Lemma (\ref{char_funct}).

However, the tuple $\tupof{t_j^M(\bar{a} \on p)}{j\le m}$ is one of the tuples of sections $\tupof{s_j}{j\le m} \in {^{m+1}|M|}$ 
over which the $m+1$-fold supremum to which we showed $ [\phi^\exists (\bar{a})]_M$ is equal is taken.

Thus \begin{eqn}\begin{split} [\phi(\bar{a})]_M = & \sss E \bar{a}\on p \wedge [t_m^M (\bar{a} \on p) = f^M( t_0^M (\bar{a} \on p), \dots, t_{m-1}^M (\bar{a} \on p))]_M \\
    = & \sss E \bar{a}\on p \wedge Et_0^M (\bar{a} \on p) \wedge \dots \wedge  Et_{m}^M (\bar{a} \on p)  \wedge \\
    & \sss \quad[t_m^M (\bar{a} \on p) = f^M( t_0^M (\bar{a} \on p), \dots, t_{m-1}^M (\bar{a} \on p))]_M \\
    \le &  \sss [\phi^\exists (\bar{a})]_M,
  \end{split}
\end{eqn}
since for each $j\le m$ we have $E\bar{a}\on p \le Et_j^M(\bar{a}\on p)$.

Case 3: Let $\phi(\bar{v})$ be $R(t_0(\bar{v}), \dots, t_{m-1}(\bar{v}))$ where for $j < m$ the $t_j (\bar{v})$ are $L$-terms, not being all variable symbols and $R$ is an
$m$-ary relation symbol. Let $M \in pSh(\Omega,L)$ and  $\bar{a}=\tup{a_0, \dots, a_{n-1}} \in {^m|M|}$. Let $p = E\phi = \bigwedge_{j \le m-1} Et_j$.

Firstly, again by \cite{BMPI}, Definitions (2.23) and (2.17), we have 
\begin{eqn}\begin{split}[\phi(\bar{a})]_M = \sss & [R^M(t_0^M (\bar{a} \on p), \dots, t_{m-1}^M (\bar{a} \on p))]_M \\
    = \sss & E \bar{a}\on p \wedge [R^M(t_0^M (\bar{a} \on p), \dots, t_{m-1}^M (\bar{a} \on p))]_M .
  \end{split}
\end{eqn}

On the other hand, by definition, $\phi^\exists (\bar{v})$ is
\[  \exists x_0 \dots \exists x_{m-1} ((x_0=   t_0(\bar{v}))^\exists \wedge \dots \wedge (x_{m-1} = t_{m-1} (\bar{v}))^\exists \wedge R(x_0, \dots, x_{m-1})).\]
For $j < m$, by induction,
$[(x_j = t_j (\bar{v}))^\exists (\bar{a})]_M = [x_j = t_j (\bar{v}) (\bar{a})]_M$.
Again, using \cite{BMPI}, Definitions (2.23) and (2.17) and the induction hypothesis, 
\begin{eqn}\begin{split}
    [\phi^\exists (\bar{a})]_M = & \bigvee_{s_0 \in |M|} \dots \bigvee_{s_{m-1} \in |M|} ( E\bar{a} \on p \wedge [s_0 \on p= t_0^M(\bar{a} \on p)] \wedge \dots \\ &
    \qquad \wedge [s_{m-1} \on p = t_{m-1}^M (\bar{a} \on p)] \wedge [R^M(s_0 \on p, \dots, s_{m-1} \on p)]_M) \\ &
    \le \sss E\bar{a} \on p \wedge [R^M (\bar{a} \on p), \dots, t_{m-1}^M (\bar{a} \on p))]_M = [\phi(\bar{a})]_M ,
  \end{split}
\end{eqn}with inequality holding by one of the defining properties of characteristic functions, cf., \cite{BMPI}, Definition (2.1).

However, the tuple $\tupof{t_j^M(\bar{a} \on p)}{j< m}$ is one of the tuples of sections $\tupof{s_j}{j<m} \in {^m|M|}$ 
over which the $m$-fold supremum to which we showed $ [\phi^\exists (\bar{a})]_M$ is equal is taken.

Thus,
\begin{eqn}\begin{split} [\phi(\bar{a})]_M = & \sss E \bar{a}\on p \wedge [R^M( t_0^M (\bar{a} \on p), \dots, t_{m-1}^M (\bar{a} \on p))]_M \\
    = & \sss E \bar{a}\on p \wedge Et_0^M (\bar{a} \on p) \wedge \dots \wedge  Et_{m-1}^M (\bar{a} \on p) \\ & \qquad \qquad \wedge [R^M( t_0^M (\bar{a} \on p), \dots, t_{m-1}^M (\bar{a} \on p))]_M \\
    \le & \sss [\phi^\exists (\bar{a})]_M ,
  \end{split}
\end{eqn}
since for each $j< m$ we have $E\bar{a}\on p \le Et_j^M(\bar{a}\on p)$.
\end{proof}

\begin{corollary}\label{evaluation_persists_through_unnesting} If $\psi$ is an $L$-formula,
  $M\in \pshol$ and $\bar{a}\in {^n|M|}$ then
  \[ [\psi(\bar{a})]_M = [\psi^\exists(\bar{a})]_M.\]
\end{corollary}

\begin{proof} By the inductive definition of the interpretation of formulae.
\end{proof}

\section{Additional connectives}\label{additional_connectives}

In this section we discuss certain extensions of the background
apparatus of our languages through adding connectives corresponding to
some or all elements of $\Omega$.  As we will see there are, in fact,
two separate, but equivalent, ways of extending by new connectives: by
additional nullary connectives or by additional unary connectives. We
explain these extensions and why they are conservative extensions of
the background apparatus of our languages.

We first of all consider introducing additional unary connectives. In
the theory of sheaves of models such extensions were used originally
to obtain results concerning the method of diagrams (\cite{BM14}) and
back-and-forth arguments (\cite{SC}, \cite{B00}).

In the context of the method of diagrams, making use of
such auxiliary connectives is arguably aesthetically pleasing. For
example, as we show at the close of this section, it allows one to recast \cite{BMPI}, Proposition
(\ref{graded_m-o-d_thm}) and Theorem (\ref{thm_m-o-diagram_equivs}) in
formats which superficially closely resemble the original results of
Robinson, with ``$\thinks$'' replaced by ``$\forces$''.  However, as
shown in \cite{BMPI}, for a satisfactory method of diagrams these extensions are not necessary.

Notwithstanding the last remark, we found it necessary to add these unary connectives
to prove the central results of the remaining sections, in particular to prove anything resembling
Proposition (\ref{prop_equiv_implies_equality}), that equivalence (for ceu formulae) up to a specified quantifier rank
between two pairs constisting of a model and a tuple from it implies the derived functions are equal. We comment further on
this in Question (\ref{question_can_we_avoid_the_squaresubs}) and the discussion after it.


\begin{definition}\label{squaresub_value_of_fml} (\emph{cf.}~\cite{BM14}, Definition (5.1)). 
  Let $p\in\Omega$, $M$ be an $L$-structure in $\psho$, $\bar{a}\in [|M|]^n$
  and $\phi(\bar{v})$ an $L$-formula with $\lh(\bar{v})=n$. Following \cite{BM14},
  write $\squaresub{p} [\phi(\bar{a})]_M$ \emph{as an abbreviation for} the element
  $E\bar{a}\land E\phi \land (p \leftrightarrow [\phi(\bar{a})]_M)$ of $\Omega$.
   
  In practice we are often able to assume that all of the constant symbols occuring in $\phi$ are interpreted by elements of $\bar{a}$, so that $E\bar{a}\le E\phi$.
  When this happens $\squaresub{p} [\phi(\bar{a})]_M$ thus abbreviates $  E\bar{a}\land (p \leftrightarrow [\phi(\bar{a})]_M)$.
\end{definition}

    We give a simple, but useful, lemma.

 \begin{lemma}\label{box_top_is_id}
 Let $L$ be a first order language with equality over $\Omega$,
  $M$ an $L$-struture, $\bar{a}\in |M|^n$ and $\phi$ an $L$-formula. Then 
 $\squaresub{\top}[\phi(\bar{a})]_M = [\phi(\bar{a})]_M$ and
 $\squaresub{\bot}[\phi(\bar{a})]_M = [\neg\phi(\bar{a})]_M$.
\end{lemma}

\begin{proof} We have \begin{eqn}\begin{split}
  \squaresub{\top}[\phi(\bar{a})]_M =  & \sss\sss E\bar{a}\land E\phi \land (\top \leftrightarrow [\phi(\bar{a})]_M) \\
 = & \sss\sss E\bar{a}\land E\phi \land [\phi(\bar{a})]_M = [\phi(\bar{a})]_M\hbox{, and, similarly,}\\
   \squaresub{\bot}[\phi(\bar{a})]_M =  &\sss\sss  E\bar{a}\land E\phi \land (\bot \leftrightarrow [\phi(\bar{a})]_M) \\
  = & \sss\sss E\bar{a}\land E\phi \land \neg[\phi(\bar{a})]_M = [\neg\phi(\bar{a})]_M.
    \end{split}
  \end{eqn}
\end{proof}

\begin{definition}\label{defn_forces_squaresub_as_abbrev} Let $p\in\Omega$, $M$ be an $L$-structure in $\psho$, $\bar{a}\in [|M|]^n$
  and $\phi(\bar{v})$ an $L$-formula with $\lh(\bar{v})=n$. 
  We go beyond \cite{BM14} and define $M\forces \squaresub{p}\phi(\bar{a})$ to be \emph{an abbreviation for}
  \[ E\bar{a}\land E\phi \le ( p \leftrightarrow [\phi(\bar{a})]_M) ,\] or, equivalently, by \cite{BMPI}, Lemma (\ref{char_equiv_in_Omega}),
  \[ E\bar{a}\land p = [\phi(\bar{a})]_M ,\] or, equivalently, 
      \[ E\bar{a}\land E\phi \land (p \leftrightarrow [\phi(\bar{a})]_M) = E\bar{a}\land E\phi.\]

  Notice that while, as shown in \cite{BMPI}, Lemma
  (\ref{interp_of_fmla_bounded_by_extents_of_fmla_and_params}),
  $[\phi(\bar{a})]_M\le E\bar{a}\land E\phi$, it is not necessarily
  the case that $( p \leftrightarrow [\phi(\bar{a})]_M) \le
  E\bar{a}\land E\phi $.
\end{definition}

  At this point \cite{BM14} extends $L$ by introducing for each $p\in \Omega$ a unary connective, which they also denote by $\squaresub{p}$. 

  \begin{definition}\label{defn_squaresub_p}
  Define $L^{\squ} = L \cup \setof{\squaresub{p}}{p\in \Omega}$, where for $p\in \Omega$ we have
  $\squaresub{p}$ is a unary connective and formula formation is extended by:
  for every $p\in \Omega$, if $\phi\in L^{\squ}$ then $\squaresub{p}\phi\in L^{\squ}$.
  It is clear that $L^{\squ}$ and $L$ have the same notions of terms and that
  $M$ is an $L^{\squ}$-structure in $\psho$ if and only if $M$ is an $L$-structure in $\psho$.

  The notions of extent and value of $L^{\squ}$ formulae are given as follows.
  If $\phi$ is an $L^{\squ}$-formula and $p\in \Omega$ then $E\squaresub{p}\phi = E\phi$.
  If $M$ is an $L$-structure in $\psho$, $\bar{a}\in |M|^n$ and $\phi(\bar{v})$ is an $L^{\squ}$-formula
  then \[
  [\squaresub{p}\phi(\bar{a})]_M =  \squaresub{p}[\phi(\bar{a})]_M \;\; ( =
  E(\bar{a}) \land E\phi \wedge ([\phi(\bar{a})]_M \leftrightarrow p) ). \]

  Extending \cite{BMPI}, Definition (\ref{defn_forces}), define
  \[ M\forces \squaresub{p}\phi(\bar{a})\hbox{ if }[\squaresub{p}\phi(\bar{a})]_M = E\squaresub{p}\phi(\bar{a})  \;\; (= E\phi\land E\bar{a} ) .\]
 This appears to be ambiguous as we have already defined $M\forces \squaresub{p}\phi(\bar{a})$, when using $\squaresub{p}\phi(\bar{a})$
  as an abbreviation for an element of $\Omega$, in Definition (\ref{defn_forces_squaresub_as_abbrev}). However, in both uses we have
  \[ M\forces \squaresub{p}\phi(\bar{a})\hbox{ if and only if }E\phi\land E\bar{a} \le ([\phi(\bar{a})]_M \leftrightarrow p),\] so the ambiguity causes no problems
  when employing ``$M\forces \squaresub{p}\phi(\bar{a})$'' in arguments.
 \end{definition}

  \begin{definition} An $L^{\squ}$ formula is a \emph{ceu formula} if it belongs to the
  closure of the unnested atomic formulae under $\Wedge$, $\exists$ and the
  $\squaresub{p}$ for $p\in \Omega$.  
\end{definition}

The results on unnesting from the previous sections extend to
${\mathcal L}^{\squ}$.

\begin{definition} If $\psi$ is an arbitrary ${\mathcal L}^{\squ}$-formula
  let $\psi^\exists$ be the formula derived from $\psi$ by for
  each atomic subformula $\phi$ replacing $\phi$ by $\phi^\exists$.
\end{definition}

\begin{definition} Let $\alpha$ be an ordinal.
  $L^{\alpha,\squ}$ is the collection of $L^{\squ}$-formulae $\phi$
  such that $r(\phi)\le \alpha$.
\end{definition}

\begin{proposition}\label{rank=degree_of_ext_squ}
If $\psi$ is an arbitrary ${\mathcal L}^{\squ}$-formula then $r(\psi)=d(\psi^\exists)$.
\end{proposition}

\begin{proof} By Proposition (\ref{rank=degree_of_ext}) and Definition (\ref{defn_squaresub_p}).
\end{proof}

\begin{proposition}
  If $\psi$ is an ${\mathcal L}^{\squ}$-formula,
  $M\in \pshol$ and $\bar{a}\in {^n|M|}$ then
  \[ [\psi(\bar{a})]_M = [\psi^\exists(\bar{a})]_M.\]
\end{proposition}

\begin{proof} By the inductive definition of the interpretation of formulae, as for, 
  Corollary (\ref{evaluation_persists_through_unnesting}), using the interpretation of $\squaresub{p}\phi$ given
  in Definition (\ref{defn_squaresub_p}).
\end{proof}

  Alternatively, one can extend $L$  by introducing for each $p\in \Omega$ a nullary connective $\check{p}$.

  \begin{definition}\label{defn_p_check} Set
  $L^{\raisebox{-0.5ex}{$\mathlarger{\check{}}$}} = L \cup \setof{\check{p}}{p\in \Omega}$, where for $p\in \Omega$ we have that
  $\check{p}$ is a nullary connective and, consequently, each $\check{p}$ is a formula and
  can be used in the inductive construction of more complex $L^{\raisebox{-0.5ex}{$\mathlarger{\check{}}$}}$-formulae.
  Again, $L^{\raisebox{-0.5ex}{$\mathlarger{\check{}}$}}$ and $L$ have the same notions of terms and
  $M$ is an $L^{\raisebox{-0.5ex}{$\mathlarger{\check{}}$}}$-structure in $\psho$ if and only if $M$ is an $L$-structure in $\psho$

  For each $p\in\Omega$ the extent and interpretation of $\check{p}$ are given by
  $E\check{p} = \top$ and $[\check{p}]_M = p$, respectively, while the extent and interpretation of formulae involving
  $\check{p}$ are given by induction on the form of the formulae.

  A variant of this approach is to only add nullary constants
  $\check{p}$ for $p$ in a specified subset of $\Omega$. Two choices
  of subsets that have proven useful are the set of dense elements of
  $\Omega$, those $p$ such that $\neg\neg p = \top$, and a dense
  subset of $\Omega$ of size $d(\Omega)$.
\end{definition}

\begin{remark} As far as we have been able to discover, the connectives $\squaresub{p}$ were introduced by Caicedo in his work
  on intuitionistic connectives in general. See \cite{Caicedo0}, \cite{Caicedo1}, \cite{Caicedo2} and \cite{CaiSig},
  as well as the previously mentioned \cite{SC}.
\end{remark}

\begin{remark} A similar definition to that of $L^{\raisebox{-0.5ex}{$\mathlarger{\check{}}$}}$ arises in fuzzy logic and continuous logic,
  for which truth values line in the real interval $[0,1]$ and 
   where including nullary connectives $\check{p}$ for either all $p\in [0,1]$ or for
   dense, countable subset of $p\in [0,1]$, satisfying the bookkeeping axioms, plays an
   important r\^ole.  See, for example, \cite{Pav} for Pavelka's logic, adding nullary connectives corresponding to each $p\in [0,1]$ and
   \cite{Haj}, \cite{BBHU} and \cite{BYU}, the former for rational Pavelka logic and the latter two for continuous logic,
   for adding nullary connectives corresponding to a countable, dense subset of $[0,1]$ (which includes the value corresponding to $\bot$, $0$ for rational Pavelka logic and
   $1$ for continuous logic). (In those traditions what
  we have called $\check{p}$ is typically denoted by $\bar{p}$, a notation we avoid in order to reduce confusion over
  whether $\bar{p}$ is a tuple of elements of $\Omega$ or a name for a nullary connective.)

  However, in \cite{BBHU}, \cite{BYU} and other work in continuous logic, the alternative of adding unary connectives (in the style of $L^{\squ}$) is often preferred.
  There, a single unary connective, written ``$\frac{x}{2}$'' (which would be written
  $\squaresub{\frac{1}{2}}$ in our notation), is included in the language and
  what would be the use of unary connectives for all other dyadic rationals are derived from expressions involving $\frac{x}{2}$ alone.
  The analogue of this manoeuvre is obviously not available in general for complete Heyting algebras $\Omega$.
\end{remark}

\begin{lemma}\label{bookkeeping} (The so-called ``bookkeeping axioms.'')
  Let $L$ be a first order language with equality over $\Omega$. Let $M$ be an $L^{\raisebox{-0.5ex}{$\mathlarger{\check{}}$}}$-struture. Let
  $p$, $q\in \Omega$. Then
  \begin{eqn}\begin{split} [\check{p}\to\check{q}]_M = &  \sss\sss p \to q = [\widecheck{p\to q}]_M\hbox{ and}\\
    [\check{p}\land\check{q}]_M = & \sss\sss p \land q = [\widecheck{p\land q}]_M.
  \end{split} 
  \end{eqn}
\end{lemma}

\begin{proof}  By the definition by induction of the interpretation of formulae we have
  \[ [\check{p}\to\check{q}] = E\check{p} \land E\check{q} \land [\check{p}]\to[\check{q}] = \top\land\top\land [\check{p}]\to[\check{q}] = p \to q
  =  [\widecheck{p\to q}]_M\] and
  \[ [\check{p}\land\check{q}] = E\check{p} \land E\check{q} \land [\check{p}]\land[\check{q}] = \top\land\top\land [\check{p}]\land[\check{q}] = p \land q
  =  [\widecheck{p\land q}]_M.\]
\end{proof}

\begin{definition} Let $L$ be a first order language with equality over $\Omega$.
  For each $L^{\squ}$-formula $\phi$  we define an $L^{\raisebox{-0.5ex}{$\mathlarger{\check{}}$}}$-formula
  $\phi^\circ$ with $E\phi = E\phi^\circ$. 
  The definition is by induction on the formation of $L^{\squ}$-formulae.
  The only interesting step is when we have
  an $L^{\squ}$-formula $\phi$ and $\phi^\circ$, and want to define define $(\squaresub{p}\phi)^\circ$.
  We set $(\squaresub{p}\phi)^\circ =  (\phi^\circ \leftrightarrow \check p)$.
  This ensures, with the middle equality being the induction hypothesis,
  that $E(\squaresub{p}\phi)^\circ = E\phi^\circ =E\phi = E(\squaresub{p}\phi) $.   
\end{definition}

\begin{proposition} If  $\phi$ is an $L^{\squ}$-formula and $\bar{a}\in |M|^n$ then  \[[\phi(\bar{a})]_M = [\phi^\circ(\bar{a})]_M. \]
\end{proposition}

\begin{proof} The proof is by induction on the formation of formulae. Again, the only interesting case is that of
  $\squaresub{p}\phi(\bar{a})$ given we have $[\phi(\bar{a})]_M = [\phi^\circ(\bar{a})]_M$ as an induction hypothesis. In this case
  \begin{eqn}\begin{split}
      [(\squaresub{p}\phi)^\circ(\bar{a})]_M & = [ (\phi^\circ(\bar{a}) \leftrightarrow \check p)]_M =
  E\bar{a} \land E\phi^\circ \land [(\phi^\circ(\bar{a})]_M \leftrightarrow [\check p]_M) \\
 & = Ea \land E\phi \land ([\phi(\bar{a})]_M \leftrightarrow p) = [\squaresub{p}\phi(\bar{a})]_M .
    \end{split}
  \end{eqn}
\end{proof}

\begin{definition} Let $L$ be a first order language with equality over $\Omega$. For each $L^{\raisebox{-0.5ex}{$\mathlarger{\check{}}$}}$-formula  $\psi$ we define
  an  $L^{\squ}$-formula $\psi^{\squ}$,  using the fact that there are tautologies.
  Let  $\tau$ be any tautology, for example, let $\tau$ be $\neg \exists v \sss v =v \lor \neg\neg \exists v\sss v=v$. Obtain
   $\psi^{\squ}$ by replacing every occurence of each $\check{p}$ in $\psi$ by $\squaresub{p}\tau$.
\end{definition}

\begin{proposition} If  $\psi$ is an $L^{\raisebox{-0.5ex}{$\mathlarger{\check{}}$}}$-formula and $\bar{a}\in |M|^n$ then  \[[\psi(\bar{a})]_M = [\phi^{\squ}(\bar{a})]_M. \]
\end{proposition}

\begin{proof}  Since
   $[\squaresub{p}\tau]_M =  ( [\tau]_M \leftrightarrow p) = (\top \leftrightarrow p) = p = [\check{p}]_M$,
   by induction on the formation of formulae we have that $[\psi(\bar{a})]_M = [\psi^{\squ}(\bar{a})]_M$.
\end{proof}

Consequently, all three approaches, using $\squaresub{p}\phi(\bar{a})$ as an abbreviation for a combination of elements of $\Omega$,
introducing unary connectives $\squaresub{p}$ and introducing nullary connectives $\check{p}$, in each case for all $p\in\Omega$,
are equivalent semantically (although different syntactically). 
\vskip6pt

We conclude this section by commenting on the relationship between our results on the method of diagrams from
\cite{BMPI}, \S{}\ref{pres_phenom_section}
and the results in \cite{BM14}, \S{}5 using the unary connectives $\squaresub{p}$ introduced in this section.
At the core of the results in \cite{BM14}, \S{}5 is the following.

\begin{proposition}\label{equiv_p=q} Let $M$ and $N$ be $L$-structures in $\psho$,
    let $\tilde{M}$ be the natural expansion of $M$ to an $L_M$-structure
    (as before) and let $\tilde{N}$ be an expansion of $N$ to an
    $L_M$-structure. Let $\phi(\bar{\underaccent{\bar}{a}})$ be an
    $L_M$-sentence and $E\bar{a} \le E\phi$, \emph{i.e.}~all of the constant symbols in $\phi$ are realized by elements of $\bar{a}$.
    Let $q=[\phi(\bar{\underaccent{\bar}{a}}^{\tilde{M}})]_{\tilde{M}}$, $p=[\phi(\bar{\underaccent{\bar}{a}}^{\tilde{N}})]_{\tilde{N}}$ and $d=p\to q$.
    \begin{eqn}\begin{split}
        & \;\; \tilde{N} \forces \squaresub{\top}\phi(\bar{\underaccent{\bar}{a}}\on q)\hbox{ and }
        \tilde{N} \forces \squaresub{d}\phi(\bar{\underaccent{\bar}{a}}^{\tilde{N}}\on p)\hbox{ if and only if }p=q.
    \end{split}
  \end{eqn}
  \end{proposition}

\begin{proof} Recall the following results from \cite{BMPI}, Lemma (\ref{forcing_preservation_lemma}).
\begin{lemma}\label{BMPI-forcing_preservation_lemma}
  Let $M$ and $N$ be $L$-structures in $\psho$,
  let $\tilde{M}$ be the natural expansion of $M$ to an $L_M$-structure, where for each $a_i\in |M|$ we have
  ${\underaccent{\bar}{a}}_i$ is a name for $a_i$,  and let $\tilde{N}$ be an expansion of $N$ to an
    $L_M$-structure. Let $\phi(\bar{\underaccent{\bar}{a}})$ be an
    $L_M$-sentence. Let $q=[\phi(\bar{\underaccent{\bar}{a}}^{\tilde{M}})]_{\tilde{M}}$, $p=[\phi(\bar{\underaccent{\bar}{a}}^{\tilde{N}})]_{\tilde{N}}$ and $d=p\to q$.
    \begin{eqn}\begin{split}
        (2)& \;\; \tilde{N} \forces \phi(\bar{\underaccent{\bar}{a}}\on q)\hbox{ if and only if } q\le p \hbox{, \emph{i.e.} }
        [\phi(\bar{\underaccent{\bar}{a}}^{\tilde{M}})]_{\tilde{M}} \le [\phi(\bar{\underaccent{\bar}{a}}^{\tilde{N}})]_{\tilde{N}}.\\
        (4)& \;\; \hbox{If  }q\le p \hbox{ then }  (E\bar{a}\on p) \land d = q = [\phi(\bar{a}\on p)]_M.
    \end{split}
  \end{eqn}
  \end{lemma}

  For ``$\Longrightarrow$'', by Lemma (\ref{box_top_is_id}),
  Lemma (\ref{BMPI-forcing_preservation_lemma}(2)) shows that the first hypothesis is equivalent to $q\le p$. On the other hand
  if $\tilde{N} \forces \squaresub{d}\phi(\bar{\underaccent{\bar}{a}}^{\tilde{N}}\on p)$ then by (one of the alternatives in) the definition
  of forcing,  $(E\bar{a}\on p)\land d = [\phi(\bar{\underaccent{\bar}{a}}^{\tilde{N}}\on p)]_{\tilde{N}}$.
  The right had side of the equality is $p\land p= p$. The left hand side, by Lemma (\ref{BMPI-forcing_preservation_lemma}(2)) and  
  Lemma (\ref{BMPI-forcing_preservation_lemma}(4)), is $p\land q$. Thus $p = p\land q$ and so $p\le q$.

  For ``$\Longleftarrow$'', use Lemma (\ref{BMPI-forcing_preservation_lemma}(2)), and  
  Lemma (\ref{BMPI-forcing_preservation_lemma}(4)) and the definition of forcing, respectively.
\end{proof}

One can then employ Proposition (\ref{equiv_p=q}) to re-write results from \cite{BMPI}, \S{}\ref{pres_phenom_section}
in the form of assertions between the equivalence of the existence of a suitable $L$-$\Gamma$-monomorphism
from $M$ to $N$ and the existence of an extension $\tilde{N}$ of $N$ which forces the relevant $L^{\squ}_M$-sentences
which $M$ forces.

\section{Equivalence of models implies equality of functions}\label{equiv_implies_equality}

In this section we show that if two pairs consisting of a model and a tuple from it agree on formulae up to a certain quantifier degree then
the functions derived from the two pairs are equal.

Here, in order to prove Proposition (\ref{prop_equiv_implies_equality}), and
consequently also in Proposition (\ref{portmanteau_prop}) where we use this proposition
 for the link `(\ref{item_equiv_ceu}) implies (\ref{item_Hs_equal})' in the circle of equivalences, 
 we need to make an assumption on the language concerning the extents of constants. 

\begin{assumption}
  In this section, \S{}\ref{portmanteau} and \S{}\ref{the_phi_alpha_Ma} we restrict to signatures
  such that $\mathcal C$ is a set and for all $c\in\mathcal C$ we have $Ec=\top$.
\end{assumption}

Without some such an assumption we might have
$E\mathcal C = \Wedge \setof{Ec}{c\in \mathcal C} = \bot$ and this would break our proof.
(Similar considerations apply in \S{}\ref{the_phi_alpha_Ma}.)\footnote{An alternative way of addresssing this problem would be by
restricting to signatures where $E\mathcal C\ne \bot$ and `working below $E\mathcal C$', proving results for
pairs of the form $(M\on E\mathcal C,\bar{a}\on E\mathcal C)$. One could arrive at this situation 
by restricting to when $E\mathcal C\ne \bot$, taking the quotient of $\Omega$ by the (principal) filter on $\Omega$ generated by $E\mathcal C$
and considering the induced pairs over the quotient. \emph{cf}.~\cite{Mir06}, Example (26.5) and Exercise (26.25).
However there seems to be no real advantage to us to doing this over simply assuming all constants have extent $\top$. },\footnote{If $\mathcal C$ is a presheaf
then,  taken literally, we always have $E\mathcal C = \bot$, since if $c\in \mathcal C$ we must have $c\on \bot \in \mathcal C$ and $Ec\on \bot = \bot$. However,
for the purposes of our proofs it would suffice to restrict attention
to signatures for which $\mathcal C$ is  a sheaf and there is some set $\mathcal D\subseteq \mathcal C$ such that $E\mathcal D\ne \bot$ and
for all $c\in \mathcal C$ there is some $d\in\mathcal D$ and some $p\in\Omega$ such that $c=d\on p$, and then work below $E\mathcal D$, as in the previous footnote.}

This consequence of this assumption is that for any formula $\phi$ in $L^{\squ}$ we $E\phi = \top$.

We start the section with a generally useful lemma. 

\begin{lemma}\label{lemma_can_focus_on_forced_formulae} Suppose $M$ is a structure in $\psh{\Omega,L}$ and $\phi(\bar{x})$ is an $L^{\squ}$ formula.
  For any $\bar{a}\in {^{\lh(\bar{x})} |M|}$ there is some $\bar{a}\in {^{\lh(\bar{x})} |M|}$ and $q\in \Omega$ such that
  \[ [\phi(\bar{a})]_M = [\phi(\bar{a}')]_M = E\bar{a}' \sss\sss\&\sss\sss \bar{a}'=\bar{a}\on q .\] 
\end{lemma}
\begin{proof} For any $\bar{a}\in {^{\lh(\bar{x})} |M|}$ we have
  \[ [ \phi(\bar{a})]_M = [ \phi(\bar{a})]_M \land [ \phi(\bar{a})]_M = [ \phi(\bar{a}\on [ \phi(\bar{a})]_M)]_M \] and
  \[ E (\bar{a}\on [ \phi(\bar{a})]_M) =  E\bar{a} \land [ \phi(\bar{a})]_M = [ \phi(\bar{a})]_M =  [ \phi(\bar{a}\on [ \phi(\bar{a})]_M)]_M .\]
  Setting $\bar{a}' = \bar{a}\on [ \phi(\bar{a})]_M$ and $q=  [ \phi(\bar{a})]_M$ gives the desired result.
\end{proof}

\begin{corollary}\label{coroll_forced_fmlae_sets} Suppose $M$ is a structure in $\psh{\Omega,L}$, $\phi(\bar{x})$ is an $L^{\squ}$ formula and $p\in \Omega$.
  \[ \setof{[ \phi(\bar{a})]_M }{ \bar{a}\in {^{\lh(\bar{x})} |M\on p|} } = \setof{ E\bar{a}}{ \bar{a}\in {^{\lh(\bar{x})} |M\on p|} \sss\sss\&\sss\sss E\bar{a} = [\phi(\bar{a})]_p } .\]
\end{corollary}

\begin{proof} By Lemma (\ref{lemma_can_focus_on_forced_formulae}), if $\bar{a}\in {^{\lh(\bar{x})} |M\on p|}$  there is some
  $\bar{a}'\in {^{\lh(\bar{x})} |M\on p|}$ such that $E\bar{a}' = [\phi(\bar{a})]_M$ and $E\bar{a}' = [\phi(\bar{a}')]_p$. So ``$\subseteq$' holds'. On the other hand,
  \begin{eqn}\begin{split} \setof{ E\bar{a}}{ \bar{a}\in {^{\lh(\bar{x})} |M\on p|} & \sss\sss\&\sss\sss E\bar{a} = [\phi(\bar{a})]_p } = \\
&  \setof{ [\phi(\bar{a})]_p }{ \bar{a}\in {^{\lh(\bar{x})} |M\on p|} \sss\sss\&\sss\sss E\bar{a} = [\phi(\bar{a})]_p } , 
    \end{split}
  \end{eqn}so ``$\supseteq$'' is also clear.
\end{proof}

\begin{definition}
  Suppose $M$ and $N$ are
  structures in $\psh{\Omega,L}$, $\bar{a}\in {^{<\lambda} |M|}$, $\bar{b}\in {^{\lh(\bar{a})} |N|}$ with
  $E\bar{a}=E\bar{b}$ and
  $h:\bar{a}\longrightarrow \bar{b}$ is given by for each $i<\lh(\bar{a})$ letting $h(a_i)=b_i$.

  Write $(M,\bar{a})\equiv^{\squ}_\alpha (N,\bar{b})$ and $(M,\bar{a})\equiv^{\squ,\ceu}_\alpha (N,\bar{b})$
  if every $L^{\alpha,\squ}$-formula $\phi$ with $r(\phi)\le\alpha$ and every ceu
  $L^{\alpha,\squ}$-formula $\phi$ with $r(\phi)\le\alpha$, respectively, is invariant under $h$ in
  $E\bar{a}$.
  
  Write
  $M\equiv^{\squ}_\alpha N$ and $M\equiv^{\squ,\ceu}_\alpha N$ if 
  $(M,\emptyset)\equiv^{\squ}_\alpha (N,\emptyset)$ and
  $(M,\emptyset)\equiv^{\squ,\ceu}_\alpha (N,\emptyset)$, respectively.
\end{definition}

\begin{proposition}\label{prop_equiv_implies_equality}
  Suppose $M$ and $N$ are
  structures in $\psh{\Omega,L}$, $\bar{a}\in {^{<\lambda} |M|}$, $\bar{b}\in {^{\lh(\bar{a})} |N|}$,
  $E\bar{a}=E\bar{b}$ and $\alpha\in \On$.
    If $(M,\bar{a})\equiv^{\squ,\ceu}_\alpha (N,\bar{b})$ then
    $H^\alpha_{M,\bar{a}} = H^\alpha_{N,\bar{b}}$. \vskip6pt
\end{proposition}

\begin{proof} We work by induction on $\alpha$. The case $\alpha=0$ and the cases for limit $\alpha$ are trivial. So, suppose the proposition holds for $\alpha$ and we have
  $(M,\bar{a})\equiv^{\squ,\ceu}_{\alpha+1} (N,\bar{b})$.

  By the definition of the functions $H^{\cdot}_{.,.}$, if $K\in\psh{\Omega,L}$,  $\bar{s}\in {^{\lh(\bar{a})} |K|}$, $\bar{t}\in {^{<\mu} |K|}$ and $\beta<\alpha$ we have
  $H^{\alpha+1}_{M,\bar{a}}(K,\bar{s}\bar{t},\beta) = H^{\beta+1}_{M,\bar{a}}(K,\bar{s}\bar{t},\beta)$ and
  $H^{\alpha+1}_{N,\bar{b}}(K,\bar{s}\bar{t},\beta) = H^{\beta+1}_{N,\bar{b}}(K,\bar{s}\bar{t},\beta)$.
  Since $(M,\bar{a})\equiv^{\squ,\ceu}_{\alpha+1} (N,\bar{b})$ we have $(M,\bar{a})\equiv^{\squ,\ceu}_{\beta+1} (N,\bar{b})$.
  Since $\beta+1\le \alpha$ the proposition holds for $\beta+1$ and we have
  $H^{\beta+1}_{M,\bar{a}}(K,\bar{s}\bar{t},\beta) = H^{\beta+1}_{N,\bar{b}}(K,\bar{s}\bar{t},\beta)$. Thus
  $H^{\alpha+1}_{M,\bar{a}}(K,\bar{s}\bar{t},\beta) = H^{\alpha+1}_{N,\bar{b}}(K,\bar{s}\bar{t},\beta)$.
  
  So what remains to be shown is that if $K\in \psh{\Omega,L}$,  $\bar{s}\in {^{\lh(\bar{a})} |K|}$ and $\bar{t}\in {^{<\mu} |K|}$, then
  $H^{\alpha+1}_{M,\bar{a}}(K,\bar{s}\bar{t},\alpha) = H^{\alpha+1}_{N,\bar{b}}(K,\bar{s}\bar{t},\alpha)$. 

  For each $\bar{c} \in  {^{\lh(\bar{t})} |M\on E\bar{t}|} $ define an equivalence relation $\sim^c$ on the formulae of $L^{\alpha,\ceu}$
  with $\lh(\bar{a})+\lh(\bar{t})$ variables by setting 
  \[ \sigma \sim^c \chi \hbox{ if and only if } [\squaresub{[\sigma(\bar{s}\bar{t})]_K} \sigma(\bar{a}\bar{c})]_M =
    [\squaresub{[\chi(\bar{s}\bar{t})]_K} \chi(\bar{a}\bar{c})]_M . \]

    Note that if $[ \squaresub{[\sigma(\bar{s}\bar{t})]_K} \sigma(\bar{a}\bar{c}) ]_M = E\bar{c}$ and $\sigma \sim^c \chi$
    then $[ \squaresub{[\chi(\bar{s}\bar{t})]_K} \chi(\bar{a}\bar{c}) ]_M = E\bar{c}$.

    For each $\bar{c} \in  {^{\lh(\bar{t})} |M\on E\bar{t}|} $ and each $\sim^c$-class $Y$ choose $\psi_{c,Y} \in Y$. Let
    \[\Psi = \setof{\psi_{c,Y}}{\bar{c} \in  {^{\lh(\bar{t})} |M\on E\bar{t}|}\sss\sss\&\sss\sss Y\hbox{ is a }\sim^c\hbox{-class}}.\]
    
  Let $\bar{x}$ be a tuple of variables with  $\lh(\bar{x}) = \lh(\bar{t})$. 

  Let $\phi(\bar{a}\bar{x})$ be the $L^{\alpha,\ceu}$-formula $\Wedge_{\psi\in \Psi} \squaresub{[\psi(\bar{s}\bar{t})]_K} \psi(\bar{a}\bar{x})$.\footnote{Here
  is where our proof may break with the ``Important Assumption'' on the extent of constants and hence formulae.}
  Applying Corollary (\ref{coroll_forced_fmlae_sets}) to $\phi(\bar{a},\bar{x})$ we have
  \[ [\exists \bar{x} \phi(\bar{a}\bar{x})]_M \land E\bar{t} =
  \Vee \setof{ E\bar{c}}{\bar{c}\in  {^{\lh(\bar{t})} |M\on E\bar{t}|} \sss\sss\&\sss\sss [\phi(\bar{a}\bar{c})]_M = E\bar{c}} .\]

  Now, for any $\bar{c}\in  {^{\lh(\bar{t})} |M\on E\bar{t}|}$ we have $[\phi(\bar{a}\bar{c})]_M = E\bar{c}$ if and only if for each $\psi\in \Psi$ we have
  $[ \squaresub{[\psi(\bar{s}\bar{t})]_K} \psi(\bar{a}\bar{c}) ]_M = E\bar{c}$. Moreover, by the remarks following Definition (\ref{defn_squaresub_p}), for each such $\psi$,
  $E \squaresub{[\psi(\bar{s}\bar{t})]_K} \psi(\bar{a}\bar{c}) = E\bar{c}$ and
  \[ [\squaresub{[\psi(\bar{s}\bar{t})]_K} \psi(\bar{a}\bar{c})]_M = E\squaresub{[\psi(\bar{s}\bar{t})]_K} \psi(\bar{a}\bar{c}) \hbox{ if and only if }
  [\psi(\bar{s}\bar{t})]_K \land E\bar{c} = [\psi(\bar{a}\bar{c})]_M .\]
  This gives us that 
  \begin{eqn}\begin{split}
      [\exists \bar{x} \sss \phi(\bar{a}\bar{x})]_M \land E\bar{t} = & \Vee\setof{E\bar{c}}{ \bar{c}\in  {^{\lh(\bar{t})} |M\on E\bar{t}|} \sss\sss\&\sss\sss \\
        & \qquad \qquad \qquad [ \Wedge_{\psi\in \Psi} \squaresub{[\psi(\bar{s}\bar{t})]_K} \psi(\bar{a}\bar{c}) ] = E\bar{c} }\\
      = & \Vee\setof{E\bar{c}}{ \bar{c}\in  {^{\lh(\bar{t})} |M\on E\bar{t}|} \sss\sss\&\sss\sss \\
        & \qquad \qquad \qquad \Wedge_{\psi\in \Psi} [\squaresub{[\psi(\bar{s}\bar{t})]_K} \psi(\bar{a}\bar{c}) ] = E\bar{c} }\\
       = & \Vee\setof{E\bar{c}}{ \bar{c}\in  {^{\lh(\bar{t})} |M\on E\bar{t}|} \sss\sss\&\sss\sss \\
        & \qquad \qquad \qquad  \forall \psi\in \Psi \sss\sss  [\psi(\bar{s}\bar{t})]_K \land E\bar{c} = [\psi(\bar{a}\bar{c})]_M } \\
      = & \Vee\setof{E\bar{c}}{ \bar{c}\in  {^{\lh(\bar{t})} |M\on E\bar{t}|} \sss\sss\&\sss\sss \\
      & \qquad \qquad \qquad  \forall \psi\in L^{\alpha,\ceu} \sss\sss  [\psi(\bar{s}\bar{t})]_K \land E\bar{c} = [\psi(\bar{a}\bar{c})]_M } \\
      = & \Vee\setof{E\bar{c}}{ \bar{c}\in  {^{\lh(\bar{t})} |M\on E\bar{t}|} \sss\sss\&\sss\sss (M,\bar{a}\bar{c})\equiv^{\squ,\ceu}_\alpha (K,\bar{s}\bar{t}\on E\bar{c}) } \\
      = & \Vee\setof{E\bar{c}}{ \bar{c}\in  {^{\lh(\bar{t})} |M\on E\bar{t}|} \sss\sss\&\sss\sss H^\alpha_{M,\bar{a}\bar{c}} = H^\alpha_{K,\bar{s}\bar{t}\on E\bar{c}} } \\
      = & H^{\alpha+1}_{M,\bar{a}}(K,\bar{s}\bar{t},\alpha), 
    \end{split}
  \end{eqn}where the first equality comes from Corollary (\ref{coroll_forced_fmlae_sets}),
  the third inequality comes from the displayed immediately above this sequence,
  the fourth equality comes from the note after the definition of the $\sim^c$,
  the fifth equality comes directly from the definition of $\equiv^{\squ,\ceu}_\alpha$ and the sixth, by induction, from the Proposition for
      $M$, $\bar{a}\bar{c}$, $K$, $\bar{s}\bar{t}\on E\bar{c}$ and $\alpha$.

      In the same way,
      $ [\exists \bar{x} \phi(\bar{b}\bar{x})]_N \land E\bar{t} = H^{\alpha+1}_{N,\bar{b}}(K,\bar{s}\bar{t},\alpha) $.

      By the hypothesis of the proposition, $(M,\bar{a}) \equiv^{\squ,\ceu}_{\alpha+1} (N,\bar{b})$. Since
      $\exists \bar{x} \sss \phi(\bar{a}\bar{x})$ is an $ L^{\alpha+1,\ceu}$-formula we have 
      $ [\exists \bar{x} \sss \phi(\bar{a}\bar{x})]_M \land E\bar{t} = [\exists \bar{x} \sss \phi(\bar{b}\bar{x})]_N \land E\bar{t}$. Hence
      $H^{\alpha+1}_{M,\bar{a}}(K,\bar{s}\bar{t},\alpha) = H^{\alpha+1}_{N,\bar{b}}(K,\bar{s}\bar{t},\alpha)$, as required.      
    \end{proof}

\section{Invariance of formulae under refined partial isomorphisms}\label{Towards_Karp's_theorem}

\begin{definition} Suppose $M$, $N \in\psh{\Omega,L}$, $h:M \longrightarrow N$ is a
  presheaf morphism, $\phi(\bar{v})$ is an $L^{\squ}$-formula and $p\in \Omega$.
  We say $\phi(\bar{v})$ is \emph{invariant under} $h$ \emph{in} $p$
  if \[ \hbox{ for all }\bar{a}\in {{}^{\lh(v)}|\dom(h)\on p|}\hbox{ we have }
  [\phi(\bar{a})]_M = [\phi(h``\bar{a})]_N .\]
  We say $\phi(\bar{v})$ is \emph{invariant under} $h$ if it is invariant under $h$ \emph{in} $\top$. 

  We also sometimes say, and this is the usage employed in \cite{BMPI}, that
  $\phi$ is \emph{preserved by} $h$ \emph{in} $p$, resp.~\emph{preserved by} $h$.
\end{definition}

 \begin{lemma}\label{pres_interp_by_big_conns} Let $p\in \Omega$.
      Suppose $h:M \longrightarrow N$ is a presheaf morphism. Let $\diamondsuit \in\set{\Wedge,\Vee}$. Let
      $\Phi$ be a collection of $L^{\squ}$-formulae each of which is invariant under $h$ in $p$.
      Then $\diamondsuit \Phi$ is invariant under $h$ in $p$.
     \end{lemma}

    \begin{proof}  Suppose $\bar{a}\in |M\on p|^n$. Let $q=\Wedge \setof{E\phi}{\phi \in \Phi}$.
         Then, using the definition of a presheaf morphism,
      \begin{eqn}
        \begin{split}
          [\diamondsuit\setof{\phi(\bar{a})}{\phi \in \Phi}]_M & = E(\bar{a} \on q) \land \diamondsuit
          [\setof{\phi(\bar{a}\on q)}{\phi \in \Phi}]_M \\
          & = E(\bar{a} \on q) \land \diamondsuit 
          [\setof{\phi(h``\bar{a}\on q)}{\phi \in \Phi}]_N \\
          & = [\diamondsuit\setof{\phi(h``\bar{a})}{\phi \in \Phi}]_M
        \end{split}
      \end{eqn}
                \end{proof}

\begin{proposition}\label{invariance_of_fml_level_by_level} Let $\alpha$ be an ordinal.
  For all $p\in \Omega$, all unnested $L^{\squ}$-formulae $\phi(\bar{v})$ with $d(\phi) = \alpha$
  and all $h\in Q_\alpha(p)$ we have $\phi$ is invariant under $h$ in $p$.
\end{proposition}

\begin{proof}
    We work by induction on quantifier degree of unnested formulae.
    Within a given degree we work by induction on the complexity of formulae.

    
    If $\phi$ is one of $\neg\psi$, $\psi \to \chi$, $\Wedge \Phi$ or $\Vee \Phi$
    where $\psi$, $\chi$ and all of the elements of $\Phi$ are invariant under $h$ at $p$
    then $\phi$ is also invariant under $h$ in $p$.
    In the first two cases $\phi$ this is by
    \cite{BMPI}, Lemma (3.6) and Lemma (3.8), respectively, and in the latter two by Lemma (\ref{pres_interp_by_big_conns}).

    If $p\in\Omega$, $\phi$ is $\squaresub{p}\psi$ and $\psi$ is invariant under $h$ at $p$ then
    \begin{eqn}\begin{split} [\phi(\bar{a})]_M & = E\psi\land E\bar{a}\land ([\phi(\bar{a})]_M \longleftrightarrow p) \\
        & = E\psi\land E\bar{b}\land ([\phi(\bar{b})]_N \longleftrightarrow p) = [\phi(\bar{b})]_N.
      \end{split}    \end{eqn}
    
    Notice also that Lemma (\ref{pres_interp_by_big_conns}) deals with the case of limit ordinal degree.
    
    Consequently, it only remains to address atomic formulae
    and formulae  of a successor rank which are formed by
    existential or universal quantification of a formulae with lesser rank.
  
If $\phi$ is an unnested atomic formula it is of one of the forms
$v_i=v_j$, $v_i = c$, $R(v_{i_0},\dots,v_{i_{n-1}})$ or $v_{i_n} = f(v_{i_0},\dots,v_{i_{n-1}})$, where the
$v_i$ are variables, $c$ is a constant symbol, $R$ is an $n$-ary relation symbol and $f$ is an $n$-ary function symbol. 

These formulae are all invariant under partial isomorphisms, either by the definition of partial isomorphism
or by Proposition (\ref{useful_prop}).

Next, we consider the case of formulae with a leading existential quantifier.
Suppose $\beta<\alpha$, $U$ is a non-empty set of variables of size less than $\lambda$
$\phi(U) = \exists V\psi(V,U)$, $d(\psi)=\beta$ (resp.~$r(\phi)=\beta$) and
$\bar{a} \in {^{\lh(U)}|M\on p|}$.

By \cite{BMPI}, Lemma (2.21) we have $[\phi(\bar{a})]_M \le E\bar{a} \land E\phi$.

As $E\bar{a}\le p$ we have\begin{eqn}
  \begin{split} [\phi(\bar{a})]_M & = [\phi(\bar{a})]_M \land p \\
    & = \Vee_{\bar{t}\in {^{\lh(V)}|M|}} [\psi(\bar{t},\bar{a})]_M\land p \\
    & = \Vee_{\bar{t}\in {^{\lh(V)}|M|}} [\psi(\bar{t}\on p,\bar{a})]_M \\
    & = \Vee_{\bar{t}\in {^{\lh(V)}|M\on p|}} [\psi(\bar{t},\bar{a})]_M .\\
  \end{split}
\end{eqn}Now we apply the Forth property for each $\bar{t}\in {^{\lh(V)}|M\on p|}$.

There is some $\setof{(h^{\bar{t}}_j,q^{\bar{t}}_j,d^{\bar{t}}_j)}{j\in J^{\bar{t}}}$ such that for all 
$\Vee_{j\in J^{\bar{t}}} q^{\bar{t}}_j = E\bar{t}$ and 
for all $j\in J^{\bar{t}}$ we have $ h^{\bar{t}}_{q_j} \subseteq h^{\bar{t}}_j$ and
$\bar{t}\on q^{\bar{t}}_j \subseteq \dom (h^{\bar{t}}_j)$.

    Applying the induction hypothesis for each $\bar{t}\on q^{\bar{t}}_j$, we have 
    \[ [\psi(\bar{t}\on q^{\bar{t}}_j,\bar{a}\on q^{\bar{t}}_j)]_M =
       [\psi(h^{\bar{t}}_j``\bar{t}\on q^{\bar{t}}_j,h^{\bar{t}}_j``\bar{a}\on q^{\bar{t}}_j)]_N .\]

\begin{eqn} 
  \begin{split}\hbox{So, }\sss\sss [\phi (\bar{a})]_M & = 
      \Vee_{\bar{t}\in {^{\lh(V)}|M\on p|}} \sss\sss \Vee_{j\in J^{\bar{t}}}\sss\sss
          [\psi(h^{\bar{t}}_j``\bar{t}\on q^{\bar{t}}_j,h^{\bar{t}}_j``\bar{a}\on q^{\bar{t}}_j)]_N  \\
   & \le
      \Vee_{\bar{s} \in {^{\lh(V)}|N\on p|}} \sss\sss \Vee_{\bar{t}\in {^{\lh(V)}|M\on p|}} \sss\sss \Vee_{j\in J^{\bar{t}}} \sss\sss
          [\psi(\bar{s} \on q^{\bar{t}}_j,h^{\bar{t}}_j``\bar{a}\on q^{\bar{t}}_j)]_N  \\
   & =
      \Vee_{\bar{s} \in {^{\lh(V)}|N\on p|}} \sss\sss \Vee_{\bar{t}\in {^{\lh(V)}|M\on p|}} \sss\sss \Vee_{j\in J^{\bar{t}}} \sss\sss
          [\psi(\bar{s} \on q^{\bar{t}}_j,h``\bar{a}\on q^{\bar{t}}_j)]_N \\
          & \qquad\qquad\qquad \hbox{ (since $h^{\bar{t}}_j$ extends $h_{q^{\bar{t}}_j}$ which comes from $h$)}\\
   & =
      \Vee_{\bar{s} \in {^{\lh(V)}|N\on p|}} \sss\sss \Vee_{\bar{t}\in {^{\lh(V)}|M\on p|}} \sss\sss \Vee_{j\in J^{\bar{t}}} \sss\sss
          ( q^{\bar{t}}_j \land [\psi(\bar{s},h``\bar{a})]_N ) \\
   & =
      \Vee_{\bar{s} \in {^{\lh(V)}|N\on p|}} \sss\sss \Vee_{\bar{t}\in {^{\lh(V)}|M\on p|}} \sss\sss \big( ( \Vee_{j\in J^{\bar{t}}} 
          q^{\bar{t}}_j) \land [\psi(\bar{s},h``\bar{a})]_N \big) \\      
   & = \Vee_{\bar{s} \in {^{\lh(V)}|N\on p|}} [\psi(\bar{s},h``\bar{a})]_N  = [\phi(h``\bar{a})]_N,
  \end{split}
    \end{eqn}where the penultimate equality holds since 
          \[ [\psi(\bar{s},h``\bar{a})]_N \le E\bar{s}
            \le Eh``\bar{t} = E\bar{t} .\]

Thus $[\phi(\bar{a})]_M \le [\phi(h``\bar{a})]_N$.

We prove $[\phi(h``\bar{a})]_N \le [\phi(\bar{a})]_M $ similarly,
using the ``Back'' property and the fact that the $h^{\bar{s}}_j$ are
one to one.

Finally, we consider the case of formulae with a leading universal quantifier.
Suppose $\beta<\alpha$, $U$ is a non-empty set of variables of size less than $\lambda$
$\phi(U) = \forall V\psi(V,U)$, $d(\psi)=\beta$ (resp.~$r(\phi)=\beta$) and
$\bar{a} \in {^{\lh(U)}|M\on p|}$.

Again, by \cite{BMPI}, Lemma (2.21) we have $[\phi(\bar{a})]_M \le E\bar{a} \land E\phi$.

As $E\bar{a}\le p$ we have\begin{eqn}
  \begin{split} [\phi(\bar{a})]_M & = [\phi(\bar{a})]_M \land p \\
    & = E\phi \land E\bar{a} \land \Wedge_{\bar{t}\in {^{\lh(V)}|M|}} ( E\bar{t} \to [\psi(\bar{t},\bar{a})]_M ) \\
    & = E\phi \land E\bar{a} \land \Wedge_{\bar{t}\in {^{\lh(V)}|M\on p|}} ( E\bar{t} \to [\psi(\bar{t},\bar{a})]_M ) \\
  \end{split}
\end{eqn}Now we apply the ``Forth'' property for each $\bar{t}\in {^{\lh(V)}|M\on p|}$.

There is some $\setof{(h^{\bar{t}}_j,q^{\bar{t}}_j,d^{\bar{t}}_j)}{j\in J^{\bar{t}}}$ such that for all 
$\Vee_{j\in J^{\bar{t}}} q^{\bar{t}}_j = E\bar{t}$ and 
for all $j\in J^{\bar{t}}$ we have $ h^{\bar{t}}_{q_j} \subseteq h^{\bar{t}}_j$ and
$\bar{t}\on q^{\bar{t}}_j \subseteq \dom (h^{\bar{t}}_j)$.

So 
\[ \Vee_{j\in J^{\bar{t}}} \sss [\psi(\bar{t}\on q^{\bar{t}}_j,\bar{a}\on q^{\bar{t}}_j)]_M 
= (\Vee_{j\in J^{\bar{t}}} \sss q^{\bar{t}}_j) \land [\psi(\bar{t},\bar{a})]_M = [\psi(\bar{t},\bar{a})]_M ,\]
since $[\psi(\bar{t},\bar{a})]_M \le E\bar{t} $.

Applying the induction hypothesis for each $\bar{t}\on q^{\bar{t}}_j$, we have 
    \[ [\psi(\bar{t}\on q^{\bar{t}}_j,\bar{a}\on q^{\bar{t}}_j)]_M =
       [\psi(h^{\bar{t}}_j``\bar{t}\on q^{\bar{t}}_j,h^{\bar{t}}_j``\bar{a}\on q^{\bar{t}}_j)]_N .\]

       Hence \[ E\bar{t} \to [\psi(\bar{t},\bar{a})]_M =
        ( \Vee_{j\in J^{\bar{t}}} E\bar{t} \on q^{\bar{t}}_j ) \to
       \Vee_{j\in J^{\bar{t}}} \sss [\psi(h^{\bar{t}}_j``\bar{t}\on q^{\bar{t}}_j,h^{\bar{t}}_j``\bar{a}\on q^{\bar{t}}_j)]_N .\]

       Thus

\begin{eqn}
  \begin{split} [\phi(\bar{a})]_M = & \sss E\psi \land E\bar{a} \land {} \\
       &  \Wedge_{\bar{t}\in {^{\lh(V)}|M\on p|}} \big( ( \Vee_{j\in J^{\bar{t}}} E\bar{t} \on q^{\bar{t}}_j ) \to
       \Vee_{j\in J^{\bar{t}}} \sss [\psi(h^{\bar{t}}_j``\bar{t}\on q^{\bar{t}}_j,h^{\bar{t}}_j``\bar{a}\on q^{\bar{t}}_j)]_N \big) \\
       \ge  & \sss E\psi \land E\bar{a} \land {} \\
       & \Wedge_{\bar{s}\in {^{\lh(V)}|N\on p|}} \big( ( \Vee_{j\in J^{\bar{s}}} E\bar{s} \on q^{\bar{t}}_j ) \to
       \Vee_{j\in J^{\bar{t}}} \sss [\psi(\bar{s}\on q^{\bar{t}}_j,h^{\bar{t}}_j``\bar{a}\on q^{\bar{t}}_j)]_N \big) \\
       = & \sss [\phi(h``\bar{a})]_N.
  \end{split}
\end{eqn}

The proof of the reverse inequality is analogous on using the ``Back'' property for ${\bar{s} \in {^{\lh(V)}|N\on p|}}$.
\end{proof}

\section{Equivalences}\label{portmanteau}

We recall that in this section we are operating under the assumption that our languages $L$ are over a signature for which
all constant symbols have extent $\top$.

\begin{proposition}\label{portmanteau_prop}
  Suppose $M$ and $N$ are
  structures in $\psh{\Omega,L}$, $\bar{a}\in {^{<\lambda} |M|}$, $\bar{b}\in {^{\lh(\bar{a})} |N|}$,
  $E\bar{a}=E\bar{b}$ and $\alpha\in \On$.
  The following are equivalent.
  \begin{enumerate}
  \item\label{item_equiv} $(M,\bar{a})\equiv^{\squ}_\alpha (N,\bar{b})$ \vskip6pt
    \item\label{item_equiv_ceu} $(M,\bar{a})\equiv^{\squ,\ceu}_\alpha (N,\bar{b})$ \vskip6pt
    \item\label{item_Hs_equal} $H^\alpha_{M,\bar{a}} = H^\alpha_{N,\bar{b}}$ \vskip6pt
    \item\label{item_Fs_are_equal} $F^\alpha_{M,\bar{a}} = F^\alpha_{N,\bar{b}}$ \vskip6pt
               \item\label{item_sim} $(M,\bar{a})\sim^{E\bar{a}}_\alpha (N,\bar{b})$ \vskip6pt
    \item\label{item_game_condit} Player II has a winning strategy for the game
      $G^\mu_\alpha(M,\bar{a},N,\bar{b},h)$ 
 \end{enumerate}
  \end{proposition}

\begin{proof} (\ref{item_equiv}) implies (\ref{item_equiv_ceu}) is immediate.
  (\ref{item_equiv_ceu}) implies (\ref{item_Hs_equal}) is Proposition (\ref{prop_equiv_implies_equality}).
  (\ref{item_Hs_equal}) is equivalent to (\ref{item_Fs_are_equal}) is Lemma (\ref{reln_Fs_and_Hs_tuples})
  (\ref{item_Fs_are_equal}) is equivalent to (\ref{item_sim}) is Proposition (\ref{equiv_fn_equal_and_Q_for_Fs_tuples}).
  (\ref{item_sim}) is equivalent to (\ref{item_game_condit}) is Theorem (\ref{thm_equiv_ws_Qs}).
    (\ref{item_sim}) implies (\ref{item_equiv}) is Proposition (\ref{invariance_of_fml_level_by_level}).
\end{proof}

\begin{corollary}\label{generalization_Karp_thm}
Suppose $M$ and $N$ are
  structures in $\psh{\Omega,L}$.
  Then $M\equiv^{\squ}_\alpha N$ if and only if there is some
  $h:M\longrightarrow N$ with $h\in Q_\alpha(\top)$.
\end{corollary}

\begin{remark}  We remark Corollary (\ref{generalization_Karp_thm})
  is a true generalization of Karp's original theorem, which can be recovered on taking
  $\Omega$ to be the two element complete Heyting algebra.
\end{remark}

\vskip12pt
\begin{corollary}\label{Karp_thm} Suppose $M$ and $N$ are
  structures in $\psh{\Omega,L_{\infty\lambda}}$.
  Then $F^\alpha_{M,\emptyset} = F^\alpha_{N,\emptyset}$ if and only if
  $M\equiv^{\squ}_\alpha N$.
\end{corollary}

We conclude with the following question.
\begin{question}\label{question_can_we_avoid_the_squaresubs}  Suppose $M$ and $N$ are
  structures in $\psh{\Omega,L}$, $\bar{a}\in {^{<\lambda} |M|}$, $\bar{b}\in {^{\lh(\bar{a})} |N|}$,
  $E\bar{a}=E\bar{b}$ and $\alpha\in \On$. Does $(M,\bar{a})\equiv_\alpha (N,\bar{b})$ imply $(M,\bar{a})\equiv^{\ceu}_\alpha (N,\bar{b})$.
\end{question}

Effectively we are asking whether the $\squaresub{p}$ are intrinsic to the proof
of Proposition (\ref{portmanteau_prop}) and more specifically of Proposition (\ref{prop_equiv_implies_equality}) or whether they are avoidable
in the sense that mentions of $\equiv^{\ceu}_\alpha$ and $\equiv^{\squ}_\alpha$ could be replaced by $\equiv_\alpha$.

We conjecture the answer to the question is in general negative.

\vskip40pt


\section{Sentences}\label{the_phi_alpha_Ma}

In this section we continue operating under the assumption that our languages $L$ are over a signature for which
all constant symbols have extent $\top$.

In  Proposition (\ref{equiv_fn_equal_and_Q_for_Fs_tuples}) and Corollary (\ref{equiv_fn_eq_and_Q_for_Hs_tuples}) 
we start with a pair $(M,\bar{a})$ and show for any pair $(N,\bar{b})$ with $E\bar{b}\le E\bar{a}$ the existence of
$Q_\alpha(E\bar{b})$ partial isomorphisms from
$\bar{a} \on E\bar{b}$ to $\bar{b}$ corresponds to equality between $F^{\alpha}_{M,\bar{a}\on E\bar{b}}$ and $F^{\alpha}_{N,\bar{b}}$ and to equality between 
$H^{\alpha}_{M,\bar{a}\on E\bar{b}}$ and $H^{\alpha}_{N,\bar{b}}$.

In Proposition (\ref{equiv_fn_equal_and_Q_for_Gs_tuples}) we have equality between (restrictions of) functions for the \emph{particular}
pair of pairs consisting of  $(M,\bar{a})$ and $(N,\bar{b})$, giving a conclusion for these two pairs alone.

In this final section we will see intimate connections between the functions defined and results proved in the \S{}\ref{function_analysis} and
infinitary sentences, analogous to those typical of the usual approaches to Scott-Karp analysis of classical models,
but involving the unary connectives $\squaresub{p}$.

We use these definitions to create for each  $M\in\pshol$, $\bar{a}\in {^{\lh(\bar{a})}|M|}$ and $\alpha\in \On$ a formula $\phi^\alpha_{M,\bar{a}}$.

\begin{definition} Let $M\in\pshol$, $\bar{a}\in {^{\lh(\bar{a})}|M|}$ and $\alpha\in \On$.
  Set
  \[ \phi^0_{M,\bar{a}} = \bigwedge_{\psi\in \UA} \squaresub{[\psi(\bar{a})]_M} \psi(\bar{v}),\sss\sss\sss
  \phi^\alpha_{M,\bar{a}} = \bigwedge_{\beta<\alpha} \phi^\beta_{M,\bar{a}} \hbox{, for limit }\alpha\hbox{, and}\]
  \[ \phi^{\alpha+1}_{M,\bar{a}} = \phi^\alpha_{M,\bar{a}} \wedge\bigwedge_{K,\bar{s}\bar{t}} \squaresub{[\exists \bar{u}\phi^\alpha_{K,\bar{s}\bar{t}}(\bar{a},\bar{u})]_M}
  \exists \bar{u}\phi^\alpha_{K,\bar{s}\bar{t}}(\bar{v},\bar{u}) .\]
\end{definition}

Note, in the corresponding definition in \cite{B00} the first conjoint in the definition of (what is here called)
$\phi^{\alpha+1}$ is $\phi^0_{M,\bar{a}}$, rather than  $\phi^\alpha_{M,\bar{a}}$ as here.
The change makes it simpler to prove equivalences such as Proposition (\ref{char_box_subscript_in_terms_of_fns_tuples}) below.

Clearly, each $\phi^\alpha_{M,\bar{a}}$ is a ceu formula.


\begin{proof} We have $[\phi^\alpha_{K,\bar{s}\bar{t}}(\bar{a}\bar{c}\on p_{\bar{c}})]_M =
  [\phi^\alpha_{K,\bar{s}\bar{t}}(\bar{a}\bar{c})]_M \land p_{\bar{c}} = p_{\bar{c}} \land p_{\bar{c}} =
  p_{\bar{c}} = E\bar{c}\on p_{\bar{c}} $.
\end{proof}

We now show equivalence between interpretation of sentences derived
from the $\phi^\alpha$ and the values of the functions introduced in
the previous section and between equality of these functions and
sentences being forced to be true. One consequence is that the
$F^\alpha$ can be defined using the $\phi^\alpha$.

\begin{proposition}\label{char_box_subscript_in_terms_of_fns_tuples} Let $M\in\pshol$, $\bar{a}\in {^{\lh(\bar{a})}|M|}$
  and $\alpha\in \On$.
  Then for any $K$, $N\in \pshol$, $\bar{b}\in {^{\lh(\bar{a})}|N|}$ with $E\bar{b}\le E\bar{a}$ and
  $\bar{s}\in {^{\lh(\bar{a})}|K|}$, $\bar{t}\in {^{<\mu}|K|}$ we have
\[ [\exists \bar{u} \phi^\alpha_{K,\bar{s}\bar{t}}(\bar{a},\bar{u})]_M = F^{\alpha+1}_{M,\bar{a}}(F^\alpha_{K,\bar{s}\bar{t}}) \] and 
\[ F^\alpha_{M,\bar{a}\on E\bar{b}} = F^\alpha_{N,\bar{b}} \hbox{ if and only if }
   [\phi^\alpha_{M,\bar{a}}(\bar{b})]_N = E\bar{b} \sss\sss (\sss \equiv N \forces \phi^\alpha_{M,\bar{a}}(\bar{b})\sss ) . \]
  
\end{proposition}

\begin{proof} The proof is by double induction on $\alpha$.
  For $\alpha=0$ and for limit $\alpha$ the proof of the second conclusion
  is immediate from the definitions.

  Suppose, next, we have for each $\bar{c}\in {^{\lh(\bar{t})}|M\on E\bar{t}|}$ the second conclusion for $M$,
  $\bar{a}\bar{c}$, $\alpha$,  $K$, and $\bar{s}\bar{t}$. Using
  this for the fourth equality and Corollary (\ref{coroll_forced_fmlae_sets}) for the second,
  we have
 \begin{eqn}\begin{split} 
     [\exists \bar{u} \phi^\alpha_{K,\bar{s}\bar{t}}(\bar{a},\bar{u})]_M = &
     \bigvee \setof{ [ \phi^\alpha_{K,\bar{s}t}(\bar{a},\bar{c})]_M}{\bar{c}\in  {^{\lh(\bar{t})}|M\on E\bar{t} |}} \\
        = & \bigvee \setof{ E\bar{c} }{\bar{c}\in  {^{\lh(\bar{t})}|M\on E\bar{t} |}
       \sss\sss\&\sss\sss [\phi^\alpha_{K,\bar{s}\bar{t}}(\bar{a}\bar{c})]_M  = E\bar{c} }  \\
     = & \bigvee \setof{ E\bar{c} }{\bar{c}\in {^{\lh(\bar{t})}|M\on E\bar{t} |}
       \sss\sss\&\sss\sss F^\alpha_{K,\bar{s}\bar{t}\on E\bar{c}} = F^\alpha_{M,\bar{a}\bar{c}} }   \\
      = & \sss F^{\alpha+1}_{M,\bar{a}}(F^\alpha_{K,\bar{s}\bar{t}})
 \end{split}  \end{eqn}and thus have the first conclusion for $M$, $\bar{a}$, $\alpha$,
 $K$ and $\bar{s}\bar{t}$.

  We thus are left to prove the second conclusion in the case of $\alpha+1$.
 Firstly, by the inductive hypothesis we have that \[ F^\alpha_{M,\bar{a}\on E\bar{b}} = F^\alpha_{N,\bar{b}} \hbox{ if and only if } [\phi^\alpha_{M,\bar{a}}(\bar{b})]_N = E\bar{b}.\]
 Secondly, fix $K\in \pshol$ and $\bar{s}\in {^{\lh(\bar{a})}|K|}$, $\bar{t}\in {^{<\mu}|K|}$. 
 We have
  \begin{eqn}
  \begin{split}
  & N \forces \squaresub{ [ \exists \bar{u} \phi^\alpha_{K,\bar{s}\bar{t}}(\bar{a},\bar{u})]_M} \exists \bar{u} \phi^\alpha_{K,\bar{s}\bar{t}}(\bar{v},\bar{u}) \hbox{ if and only if }\\
      & [ \exists \bar{u} \phi^\alpha_{K,\bar{s}\bar{t}}(\bar{a},\bar{u})]_M \land E\bar{b} = [ \exists \bar{u} \phi^\alpha_{K,\bar{s}\bar{t}}(\bar{b},\bar{u})]_N \hbox{ if and only if } \\
    & F^{\alpha+1}_{M,\bar{a}\on E\bar{b}}(F^\alpha_{K,\bar{s}\bar{t}},\alpha) = F^{\alpha+1}_{N,\bar{b}}(F^\alpha_{K,\bar{s}\bar{t}},\alpha) 
  \end{split}
  \end{eqn}where the first equivalence comes directly from the definition of
 the interpretation of formulae and and the second equivalence comes from 
 the first conclusion for both $M$, $\bar{a}$ and $\alpha$ and $N$, $\bar{b}$ and $\alpha$.
  \end{proof}

Proposition (\ref{char_box_subscript_in_terms_of_fns_tuples}) allows us to add another equivalent to the list in Proposition (\ref{portmanteau_prop}). We state
this in two ways.

\begin{corollary}\label{phis_and_functions_equiv}
    Let $M$, $N\in \pshol$, $\bar{a}\in {^{\lh(\bar{a})}|M|}$, $\bar{b}\in {^{\lh(\bar{a})}|N|}$ with $E\bar{b}= E\bar{a}$ and $\alpha\in \On$.
  \begin{enumerate}
    \item\label{item_sentences_eva_equal} $[\phi^\alpha_{M,\bar{a}}(\bar{b})]_N$ = $[\phi^\alpha_{N,\bar{b}}(\bar{b})]_N$ \vskip6pt
    \item\label{item_Fs_equal} $F^\alpha_{M,\bar{a}} = F^\alpha_{N,\bar{b}}$ \vskip6pt
  \end{enumerate}
\end{corollary}

\begin{corollary}\label{3.2.4_ish_tuples} (\emph{cf.} \cite{B00}, Theorem (3.2.4).)
  Let $M\in \pshol$, $\bar{a}\in {^{\lh(\bar{a})}|M|}$ and $\alpha\in \On$.
Then for any $N\in \pshol$ and $\bar{b}\in {^{\lh(\bar{a})}|N|}$, with $E\bar{b}\le E\bar{a}$,
$N\forces \phi^\alpha_{M,\bar{a}}(\bar{b})$ if and only if there is some $h:\bar{a}\longrightarrow \bar{b}$ with $h\in Q_\alpha(E\bar{b})$.
\end{corollary}

We make use of the following observation in the corollary below it.

\begin{corollary}\label{corol_self_app_forces}  Let $M\in \pshol$, $\bar{a}\in {^{\lh(\bar{a})}|M|}$ and $\alpha\in \On$. Then
 \[  [\phi^\alpha_{M,\bar{a}}(\bar{a})]_M = E\bar{a} \sss\sss (\sss \equiv M \forces \phi^\alpha_{M,\bar{a}}(\bar{a}) \sss ) . \]
\end{corollary}

\begin{proof} Immediate on setting $(N,\bar{b})=(M,\bar{a})$ in the second conclusion of Proposition (\ref{char_box_subscript_in_terms_of_fns_tuples}).
\end{proof}

\begin{corollary} Let $\psi\in L^{\squ}$ be a sentence with $r(\psi)=\alpha$. Let $M\in \pshol$. Then
  $[\psi]_M = \top$ if and only if \[ [\Vee \setof{\phi^\alpha_{N,\emptyset}}{N\in \pshol \sss\sss\&\sss\sss [\psi]_N=\top}]_M =\top .\]
\end{corollary}

\begin{proof} ``$\Longrightarrow$.'' By Corollary (\ref{corol_self_app_forces}),  $[\phi^\alpha_{M,\emptyset}]_M =\top$. So, if $[\psi]_M = \top$,
  then
  \begin{eqn}\begin{split}
      \top = [\phi^\alpha_{M,\emptyset}]_M \le  & \sss
      \Vee \setof{[\phi^\alpha_{N,\emptyset}]_M}{N\in \pshol \sss\sss\&\sss\sss [\psi]_N=\top} \\
      =   &     [\Vee \setof{\phi^\alpha_{N,\emptyset}}{N\in \pshol \sss\sss\&\sss\sss [\psi]_N}=\top]_M \le \top,
    \end{split}
  \end{eqn}so $[\Vee \setof{\phi^\alpha_{N,\emptyset}}{N\in \pshol \sss\sss\&\sss\sss [\psi]_N}=\top]_M = \top$.

  ``$\Longleftarrow$." Now suppose $N\in \pshol$ with $[\psi]_N=\top$ and set $p_N =  [\phi^{\alpha}_{N,\emptyset}]_M$.
 
  As $[\phi^\alpha_{N,\emptyset}]_N = \top$, we have $[\phi^\alpha_{N,\emptyset}]_{N\on p_N} = p_N$. On the other hand,
   \[ [\phi^{\alpha}_{N,\emptyset}]_{M\on p_N} = [\phi^{\alpha}_{N,\emptyset}]_M \land p_N = p_N \land p_N = p_N .\]

  By Corollary (\ref{3.2.4_ish_tuples}) there is thus some $h:N\on p_N \longrightarrow M\on p_N$ such that  $h\in Q_\alpha(p_N)$.
  Hence, by Corollary (\ref{generalization_Karp_thm}) we have $N\on p_N \equiv^{\squ}_\alpha M\on p_N$.
  So, since $[\psi]_{N\on p_N} = p_N$ we have $[\psi]_{M\on p_N} = p_N$.

  Thus  if \begin{eqn}\begin{split}
    \top =  & \sss[\Vee \setof{\phi^\alpha_{N,\emptyset}}{N\in \pshol \sss\sss\&\sss\sss [\psi]_N}=\top]_M \\
      = &     \Vee \setof{[\phi^\alpha_{N,\emptyset}]_M}{N\in \pshol \sss\sss\&\sss\sss [\psi]_N=\top}.
    \end{split}
  \end{eqn}we have\begin{eqn}\begin{split} \top  
                                     = & \Vee \setof{p_N}{N\in \pshol \sss\sss\&\sss\sss [\psi]_N=\top} \\
       = & \Vee \setof{[\psi]_{M\on p_N}}{N\in \pshol \sss\sss\&\sss\sss [\psi]_N=\top} \le [\psi]_M \le \top,\\
    \end{split}
  \end{eqn}so $[\psi]_M=\top$.
  \end{proof}

Proposition (\ref{char_box_subscript_in_terms_of_fns_tuples}) also allows us to see that the sentences $\phi^\alpha$ can be defined in terms of the functions $F^\alpha$.

\begin{proposition} Let $M\in\pshol$, $\bar{a}\in {^{\lh(\bar{a})}|M|}$ and $\alpha\in \On$. Then
  $\phi^0_{M\bar{a}} = \bigwedge_{\psi\in \UA} \squaresub{F^0_{M,\bar{a}}} \psi(\bar{v})$, and
  \[ \phi^{\alpha+1}_{M,\bar{a}} = \phi^\alpha_{M,\bar{a}} \wedge\bigwedge_{K,\bar{s}\bar{t}}
  \squaresub{ F^{\alpha+1}_{M,\bar{a}} (K,\bar{s}\bar{t},\alpha) } \exists u\phi^\alpha_{K,\bar{s}\bar{t}}(\bar{v},\bar{u}) .\]
\end{proposition}

\begin{proof} For $\alpha=0$ this is immediate from the definitions.
  For $\alpha+1$ it is immediate from the definitions and Proposition (\ref{char_box_subscript_in_terms_of_fns_tuples}).
\end{proof}

\vfill\eject

\end{document}